\documentclass[a4paper]{amsart}
\usepackage{fix-cm}
\usepackage[final]{microtype}
\usepackage[dvipsnames,svgnames,x11names,hyperref]{xcolor}
\usepackage{dsfont,url,graphicx,verbatim,amssymb,enumerate,stmaryrd,booktabs,lmodern,mathtools,mathabx,nicefrac}
\SetSymbolFont{stmry}{bold}{U}{stmry}{m}{n}
\usepackage[pagebackref,colorlinks,citecolor=Mahogany,linkcolor=Mahogany,urlcolor=Mahogany,filecolor=Mahogany]{hyperref}
\usepackage[capitalize]{cleveref}
\usepackage[mathscr]{euscript}
\usepackage[margin=3cm]{geometry}
\usepackage{tikz,tikz-cd}
\usetikzlibrary{matrix,calc,positioning,arrows,decorations.pathreplacing,patterns,arrows,patterns.meta}

\newtheorem{theorem}{Theorem}[section]
\newtheorem*{theorem*}{Theorem}
\newtheorem{lemma}[theorem]{Lemma}
\newtheorem{proposition}[theorem]{Proposition}

\newtheorem*{corollary*}{Corollary}

\newtheorem{claim}[theorem]{Claim}
\newtheorem{atheorem}{Theorem}

\newtheorem{innercustomgeneric}{\customgenericname}
\providecommand{\customgenericname}{}
\newcommand{\newcustomtheorem}[2]{%
  \newenvironment{#1}[1]
  {%
   \ifdefined\crefalias\crefalias{innercustomgeneric}{#2}\fi
   \renewcommand\customgenericname{#2}%
   \renewcommand\theinnercustomgeneric{##1}%
   \innercustomgeneric
  }
  {\endinnercustomgeneric}%
  \ifdefined\crefname\crefname{#2}{#2}{#2s}\fi
}
\newcustomtheorem{customthm}{Theorem}

\newcustomtheorem{customconj}{Conjecture}
\theoremstyle{definition}
\newtheorem{definition}[theorem]{Definition}
\newtheorem*{definition*}{Definition}

\newtheorem{notation}[theorem]{Notation}

\theoremstyle{remark}
\newtheorem{example}[theorem]{Example}
\newtheorem*{example*}{Example}
\newtheorem*{remark*}{Remark}

\newtheorem{remark}[theorem]{Remark}

\usepackage[textwidth=25mm, textsize=tiny]{todonotes}

\crefname{lemma}{Lemma}{Lemmas}
\crefname{theorem}{Theorem}{Theorems}
\crefname{definition}{Definition}{Definitions}
\crefname{proposition}{Proposition}{Propositions}
\crefname{remark}{Remark}{Remarks}
\crefname{corollary}{Corollary}{Corollaries}
\crefname{equation}{Equation}{Equations}
\crefname{construction}{Construction}{Constructions}
\crefname{ex}{Example}{Examples}
\crefname{appsec}{Appendix}{Appendices}
\crefname{subsection}{Subsection}{Subsections}

\newcommand{\bbP}{\mathbb{P}}

\newcommand{\bbZ}{\mathbb{Z}}

\newcommand{\tensor}{\otimes}
\newcommand{\St}{\mathrm{St}}

\newcommand{\SL}{\mathrm{SL}}
\newcommand{\GL}{\mathrm{GL}}
\newcommand{\BGL}{\mathrm{BGL}}
\newcommand{\F}{\mathbb{F}}
\newcommand{\gen}[1]{\langle #1\rangle}

\DeclareMathOperator{\id}{id}

\DeclareMathOperator{\GW}{GW}

\newcommand{\compactldots}{\mathinner{\ldotp\mkern-2mu\ldotp\mkern-2mu\ldotp}}

\newcommand{\scr}[1]{{\mathscr{#1}}}
\renewcommand{\rm}[1]{{\mathrm{#1}}}
\newcommand{\ul}[1]{{\underline{#1}}}
\newcommand{\ol}[1]{{\overline{#1}}}

\newcommand{\fr}[1]{{\mathfrak{#1}}}
\newcommand{\bb}[1]{{\mathds{#1}}}

\renewcommand{\bf}[1]{{\mathbf{#1}}}

\AtBeginDocument{%
	\def\MR#1{}
}

\title{Double Steinberg coinvariants for special linear groups}

\author[T. Abdelnaim]{Tatiana Abdelnaim}
\address{Department of Mathematics, University of Oklahoma, 601 Elm Ave Rm 423, Norman, OK 73019-3103, USA}
\email{tatiana.a.abdelnaim-1@ou.edu}
\author[D. Chan]{David Chan}
\address{Department of Mathematics, Michigan State University, 619 Red Cedar Road, East Lansing, MI 48823, USA}
\email{chandav2@msu.edu}
\author[A. Kupers]{Alexander Kupers}
\address{Department of Computer and Mathematical Sciences, University of Toronto Scarborough, 1265 Military Trail, Toronto, ON M1C 1A4, Canada}
\email{a.kupers@utoronto.ca}
\author[R. J. Sroka]{Robin J.\ Sroka}
\address{Mathematisches Institut, Universität Münster, Einsteinstrasse 62, 48149 Münster, Germany}
\email{robinjsroka@uni-muenster.de}
\author[M. Scalamandre]{Matthew Scalamandre}
\address{Department of Computer and Mathematical Sciences, University of Toronto Scarborough, 1265 Military Trail, Toronto, ON M1C 1A4, Canada}
\email{m.scalamandre@utoronto.ca}

\begin{document}

\begin{abstract}
    For a field $F$ we study the coinvariants for the $\SL_n(F)$-action on the double Steinberg module $\St_n(F) \otimes \St_n(F)$ and show they have a rich algebraic structure: for $n = 2$ it is the Grothendieck--Witt group of $F$, and for all $n$ they assemble to a graded nonunital $\bbZ[F^\times]$-algebra, whose rational (underived) indecomposables may be expressed in terms of the augmentation ideal of the Grothendieck--Witt group. We then explain, building on work of Galatius--Kupers--Randal-Williams, that the special linear groups $\SL_n(F)$ assemble to an $E_\infty$-algebra in a suitable functor category, whose $E_2$-homology groups have a vanishing line of slope 2 and on the critical line are given by the double Steinberg coinvariants.
\end{abstract}

\maketitle

\vspace{-.5cm}

\tableofcontents

\vspace{-.5cm}

\section{Introduction}

Let $F$ be a field. Quillen observed that when one studies the group homology $H_*(\GL_n(F))$ of general linear groups over $F$, or its algebraic $K$-theory groups $K_*(F)$, a crucial role is played by the homology groups $H_*(\GL_n(F);\St_n(F))$ with coefficients in the so-called \emph{Steinberg module}, defined as the reduced top homology group of the Tits building $T(F^n)$ \cite{Quillen}. Recent work of Galatius, Kupers, and Randal-Williams \cite{GKRW23,GKRW25} found that the \emph{double Steinberg module}
\[\St^2_n(F) \coloneq \St_n(F) \otimes \St_n(F)\]
and the homology groups $H_*(\GL_n(F);\St^2_n(F))$ are similarly important, as long as $F$ is an infinite field. They moreover established that the $\GL_n(F)$-coinvariants of $\St^2_n(F)$ are isomorphic to $\bb{Z}$ for all $n$, that is, $H_0(\GL_n(F);\St_n^2(F)) \cong \bb{Z}$. More recently, $H_1(\GL_n(F);\St_n^2(F))$ was described in terms of multiple polylogarithms \cite{KRS1}.

\medskip

In this paper we consider analogous questions for the \emph{special linear groups} $\SL_n(F)$ and initiate the study of its homology with coefficients in the double Steinberg module,
\[
    H_*(\SL_n(F);\St_n^2(F)).
\]
Our first result is a computation of the $\SL_n(F)$-coinvariants of $\St^2_n(F)$ for $n=2$ and our answer is phrased in terms of the (symmetric) \emph{Grothendieck--Witt group} $\rm{GW}(F)$: this is the group completion of the abelian monoid of isomorphism classes of nondegenerate symmetric bilinear forms over $F$ under orthogonal sum (see e.g.\ \cite[Definition 2.1.9]{Scharlau} and \cite[Definition 2.1.2]{deglise}), and has a presentation with generators the rank one forms $\langle a \rangle$  for $a \in F^\times$ given by $(x,y) \mapsto axy$, and relations (i) $\langle\lambda^2 a \rangle = \langle a \rangle$ for $a,\lambda \in F^\times$ and (ii) $\langle a \rangle + \langle b \rangle = \langle a+b\rangle + \langle ab(a+b) \rangle$ for $a,b,a+b \in F^\times$ (see e.g.\ \cite[Theorem 2.9.4]{Scharlau} and \cite[Theorem 2.1.11]{deglise}).

\begin{atheorem}\label{mainthm:dimension 2} There is an isomorphism $\St^2_2(F)_{\SL_2(F)} \cong \GW(F)$.
\end{atheorem}

\begin{example}It follows from \cref{mainthm:dimension 2} that $\St^2_2(F)_{\SL_2(F)}$ is a quotient of $\bb{Z}[F^\times/(F^\times)^2]$. If $F$ contains a square root of $-1$ then $\rm{GW}(F) \cong \bb{Z} \oplus \{\text{$2$-torsion}\}$; for example, one may use the equation $\langle a \rangle+\langle -a \rangle = \langle 1 \rangle+\langle -1 \rangle$ \cite[Theorem 2.1.11]{deglise}, which in this case implies $2(\langle a \rangle-\langle 1\rangle)=0$.
\end{example}

\begin{remark*}It is plausible that through the mechanism of \cref{sec:ek-cells}, \cref{mainthm:dimension 2} is related to the description of $H_2(\rm{SL}_2(F);\bb{Z})$ as a pullback $\smash{\rm{I}^2(F) \times_{\rm{I}^2(F)/\rm{I}^3(F)} K_2(F)}$ involving the powers of the augmentation ideal $\rm{I}(F)$ of $\rm{GW}(F)$ \cite[Theorem 6.5]{Suslin} \cite[Theorem 10]{Mazzoleni}.\end{remark*}

For $n>2$ our results are less complete. We recover certain features of $\rm{GW}(F)$, but we do not go as far as to suggest that $\St^2_n(F)_{\SL_n(F)}$ for $n>2$ be thought of as a generalization of the Grothendieck--Witt group, as we do not give an interpretation of it in terms of forms.

\begin{atheorem}\label{mainthm:higherdim} \,
\begin{enumerate}[(i)]
    \item \label{enum:higherdim-i} $\St^2_n(F)_{\SL_n(F)}$ is a quotient of $\bb{Z}[F^\times/(F^\times)^n]$.
    \item \label{enum:higherdim-ii} If $F$ contains an $n$th root of $-1$ then $\St^2_n(F)_{\SL_n(F)}$ is of the form $\bb{Z} \oplus \{\text{$n$-torsion}\}$.
\end{enumerate}
\end{atheorem}

We obtain further information after rationalization and passing to indecomposables with respect to the following product: direct sum induces a product on $\St(F) \coloneq \bigoplus_n \St_n(F)$, and $\St^2(F) \coloneq \bigoplus_n \smash{\St^2_n(F)}$ inherits a product by using this termwise, see e.g.~\cite[Section 6.3]{GKRW25}. This in turn induces a product on 
\[\St^2(F)_{\SL(F)} \coloneq \bigoplus_n \smash{\St^2_n(F)_{\SL_n(F)}}.\] 
Rationally, we compute its (underived) indecomposables $Q\big[\St^2(F)_{\SL(F)}\big]$ in terms of the augmentation ideal $\rm{I}(F)$ of the Grothendieck--Witt group, i.e.\ the group $\rm{I}(F)$ is the kernel of the map $\rm{GW}(F) \to \bb{Z}$ induced by sending a form to its rank.

\begin{atheorem}\label{mainthm:rational} The (underived) indecomposables of the algebra $\St^2(F)_{\SL(F)}\otimes\bb{Q}$ are given by
    \[Q\big[\St^2(F)_{\SL(F)}\big]_n \otimes \bb{Q} \cong \begin{cases}
        \bb{Q} & \text{if $n=0,1$,} \\ 
        \rm{I}(F)\otimes \bb{Q} & \text{if $n=2$,} \\ 
        0 & \text{otherwise.}
    \end{cases}\]
\end{atheorem}

We remark that $\rm{I}(F)\otimes \bb{Q}$ is isomorphic to the rationalized \emph{Witt group} $\rm{W}(F) \otimes \bb{Q}$, where $\rm{W}(F) = \rm{GW}(F)/(h)$ is the quotient of $\rm{GW}(F)$ by the isomorphism class of the hyperbolic form $h = \langle 1 \rangle + \langle -1 \rangle$ (see e.g.\ \cite[Definition 2.1.9]{Scharlau} and \cite[Definition 2.1.6]{deglise}).

\begin{example}If $\sqrt{-1}\in F$ then the (underived) indecomposables of $\St^2(F)_{\SL(F)} \otimes \bb{Q}$ are given by $\bb{Q}$ when $n=0,1$ and $0$ otherwise, because $\rm{W}(F) \otimes \bb{Q}$ is 1-dimensional.
\end{example}

In \cref{sec:ek-cells} we will provide an interpretation of these computations. Recall that the double Steinberg module appears in the work of Galatius, Kupers, and Randal-Williams through a study of structure of the $E_\infty$-algebra $\mathbf{BGL}(F)$, given by the disjoint union $\bigsqcup_n \BGL_n(F)$ with product induced by block sum, \emph{considered as an $E_2$-algebra} \cite{GKRW23,GKRW25}. Namely, for an infinite field $F$ they establish an isomorphism $\smash{H^{E_2}_{n,d}(\mathbf{BGL}(F))} \cong \smash{H_{d-(2n-2)}(\GL_n(F);\St^2_n(F))}$ between its $E_2$-homology groups, in the sense of \cite{GKRW23}, and the homology of general linear groups with coefficients in the double Steinberg module. Thus their coinvariant computation determines the first nonzero $E_2$-homology in each rank $n$. We give a similar interpretation of $H_*(\SL_n(F);\St^2_n(F))$, so in particular of the coinvariants, in terms of an $E_\infty$-algebra $\mathbf{BSL}(F)$ constructed out of the classifying spaces of the special linear groups in a symmetric monoidal category of graded lines:

\begin{atheorem}
    \label{mainthm:ekcells}
    If $F$ is an infinite field, then the $E_2$-homology of $\bf{BSL}(F)$ is given by
    \[H^{E_2}_{n,d}(\bf{BSL}(F)) \cong \widetilde{H}_{d-(2n-2)}(\SL_n(F);\St^2_n(F)).\]
    In particular, the first nonzero groups are $H^{E_2}_{n,2n-2}(\bf{BSL}(F)) \cong \St^2_n(F)_{\SL_n(F)}$.
\end{atheorem}

In principle this allows one to obtain results for the homology of special linear groups analogous to those for general linear groups in \cite{GKRW25}, e.g.~recover the results of \cite{Schlichting}. However, this seems to be involved and we leave it to future work.

\subsection*{Acknowledgments} The authors would like to thank the organizers of the \emph{Collaborative Research Workshop on $K$-theory and Scissors Congruence} at Vanderbilt University in July 2024, during which this paper was initiated, and the conference \emph{Scissors congruence and $K$-theory} at the University of Pennsylvania in July 2025, during which some of the results were presented. We would also like to thank Thor Wittich for helpful conversations. TA was partially supported by NSF grant DMS-2405310 as well as a Simons Foundation Travel Support for Mathematicians grant. DC was partially supported by NSF grant DMS-2135960. AK acknowledges the support of the Natural Sciences and Engineering Research Council of Canada (NSERC) [funding reference number 512156 and 512250]. RJS was supported by the German Research Foundation through SFB 1442 -- 427320536, Geometry: Deformations and Rigidity, and EXC 2044 -- 390685587, Mathematics M\"unster: Dynamics--Geometry--Structure.

\section{Double Steinberg coinvariants} In this section we begin our study of the $\SL_n(F)$-coinvariants of the double Steinberg module. In particular, we will construct increasingly small generating sets, eventually only requiring a single generator for each rank when considering all coinvariant groups as a graded $\bb{Z}[F^\times]$-module.

\begin{notation} We fix a field $F$ and write $\St_n$ instead of $\St_n(F)$, and $\SL_n$ instead of $\SL_n(F)$.\end{notation}

Recall from the introduction that the double Steinberg module is defined as
\[\St^2_n \coloneq \St_n\otimes\St_n.\]

Our goal is to understand the $\SL_n$-coinvariants of $\St^2_n$, for which we introduce an abbreviated notation:

\begin{notation}We write $\scr{A}_n \coloneq (\St^2_n)_{\SL_n} = (\St_n\otimes\St_n)_{\SL_n}$.
\end{notation}

\subsection{Generating $\scr{A}_n$ as an abelian group} We begin by constructing increasingly smaller generating sets of $\scr{A}_n$ as an abelian group. For this purpose, we consider the Steinberg module as defined through the presentation in \cite[Section 2.1]{GKRW25} (that is, the Bykovskii presentation \cite{Bykovski}):

\begin{definition}\label{def:steinberg} The \emph{Steinberg module} $\St_n$ is given by the quotient of the free abelian group 
\[\St_n \coloneq \frac{\bb{Z}\big\{[A] \mid \text{$[A] \in \rm{GL}_n(F)$}\big\}}{\text{\eqref{enum:steinberg-1}--\eqref{enum:steinberg-3}}}\]
by the following relations, writing $[A] = [v_1,\compactldots,v_n]$ when the $i$th column of $A$ is given by $v_i$:
\begin{enumerate}[(a)]
    \item \label{enum:steinberg-1} (Column addition) For any $\lambda \in F^\times$ and $1 \leq i<j \leq n$ we have
    \[[v_1,\compactldots,v_n] = [v_1,\compactldots, v_i + \lambda v_j,\compactldots,v_j,\compactldots,v_n]+[v_1,\compactldots,v_i,\compactldots, v_i + \lambda v_j,\compactldots,v_n].\]
    \item \label{enum:steinberg-2} (Column scaling) For any $\lambda \in F^\times$ and $1 \leq i \leq n$ we have
    \[[v_1,\compactldots,v_n] = [v_1,\compactldots,\lambda v_i,\compactldots,v_n].\]
    \item \label{enum:steinberg-3} (Column permutation) For any $\sigma \in \fr{S}_n$ we have
    \[[v_1,\compactldots,v_n] = (-1)^\sigma [v_{\sigma(1)},\compactldots,v_{\sigma(n)}].\]
\end{enumerate}
This has a natural $\GL_n(F)$-action by left matrix multiplication $M \cdot [A] = [MA]$.
\end{definition}

Since the double Steinberg module is defined as
$\St^2_n = \St_n\otimes\St_n$,
we can think of its elements as finite sums of basic tensors of the form $[A]\otimes [B]$, where $A,B\in \GL_n(F)$. This has an induced $\GL_n(F)$-action given by $M\cdot([A]\otimes[B]) = [MA]\otimes [MB]$ and $\scr{A}_n$ is given by coinvariants with respect to the subgroup $\SL_n \subseteq \GL_n$. Combined with the presentation above, we obtain the following first generating set for $\scr{A}_{n}$:

\begin{lemma}\label{lem:genset} $\scr{A}_{n}$ is generated by elements of the form $[\rm{id}_n]\otimes [M]$ where $M\in \GL_n(F)$.
\end{lemma}
\begin{proof}
    It suffices to show that every element $[A]\otimes [B]$ can be identified with one of the form $[\rm{id}_n]\otimes [M]$. By \cref{def:steinberg} \eqref{enum:steinberg-2} we have $[A] = [A']$ whenever $[A']$ is obtained from $[A]$ via column scaling. Thus, if $A'$ is obtained from $A$ by scaling the first column by $\det(A)^{-1}$ we obtain $[A]\otimes[B] = [A']\otimes [B]$ where $[A']\in \SL_n(\F)$. Acting by $(A')^{-1}$ proves the claim. 
\end{proof}

With an eye towards reducing the size of this generating set, we introduce abbreviations for its elements and fix notation for generators that will be key in the reduction process:

\begin{definition}\,
\begin{enumerate}[\noindent (i)]
    \item For $M\in \GL_n$ we write $\gen{M}$ for the element $[\rm{id}_n]\otimes[M]\in \scr{A}_n$.
    \item For a basis $v_1,\ldots v_n$ for $F^n$ we write $\langle v_1,\compactldots,v_n \rangle$ for the class of $\gen{M}$, where $M$ is the matrix with $i$th column $v_i$.
    \item For $\lambda\in F^{\times}$ and $n \geq 2$ we write $\{\lambda\}_n$ for the class of the matrix with $1$'s on the diagonal and superdiagonal, except that the $(n-1,n)$ entry is $\lambda$, and with every other entry equal to zero. For $\lambda\in F^{\times}$ and $n = 1$, we simply write $\{\lambda\}_1$ for the class of $\rm{id}_1$.
\end{enumerate}
\end{definition}

\begin{notation}
    \label{nota:diag-sdiag}
    We will write $\rm{diag}(a_1,\compactldots,a_n)$ for the diagonal $(n \times n)$-matrix with diagonal entries $a_1,\ldots,a_n \in F$. We will write $\rm{sdiag}(\lambda_1,\compactldots,\lambda_{n-1})$ for the superdiagonal $(n \times n)$-matrix with $1$'s on the diagonal and entries $\lambda_1,\ldots,\lambda_{n-1}$ just above the diagonal.
\end{notation}

\begin{example} For instance, $\{\lambda\}_4$ is the class of the following matrix:
    \[
       \rm{sdiag}(1,1,\lambda) = \begin{bmatrix}
            1 & 1 & 0 & 0\\
            0 & 1 & 1 & 0\\
            0 & 0 & 1 & \lambda\\
            0 & 0 & 0 & 1\\
        \end{bmatrix}.
    \]
\end{example}

It follows from \cref{def:steinberg} that the elements $\langle M \rangle$ for $M\in \GL_n$ generating $\scr{A}_n$ satisfy several relations related to row and column operations. We will rephrase these more concisely using the following notation: (i) $e_{ij}(\lambda)$ denotes the matrix which is equal to the identity except for the $(i,j)$ position, where it is equal to $\lambda$, (ii) a monomial matrix is a square matrix with exactly one nonzero entry in each row and in each column and it has an associated permutation matrix  given by replacing these entries with $1$.

\begin{lemma}\label{lem:an-relations}
    The classes $\langle M\rangle = \langle v_1,\compactldots,v_n\rangle$ of $\scr{A}_n$ satisfy the following relations.
\begin{enumerate}[\noindent (1)]
    \item \label{enum:an-relations-i} (Column addition) For any $\lambda \in F^{\times}$ and $i \neq j$ we have 
    \[\langle M \rangle = \langle M e_{ij}(\lambda) \rangle+\langle Me_{ji}(\lambda^{-1})) \rangle.\]
    \item \label{enum:an-relations-ii} (Column scaling and permutation) Let $N$ be a monomial matrix in $\GL_n$, and $\sigma \in \fr{S}_n$ be the associated permutation. Then we have
    \[\langle M \rangle = (-1)^\sigma \langle MN \rangle.\]
    \item \label{enum:an-relations-iii} (Row addition) For any $\lambda\in F^{\times}$ and $i\neq j$ we have 
    \[
        \langle M\rangle = \langle e_{ij}(\lambda)M\rangle+\langle e_{ji}(\lambda^{-1})M\rangle.
    \]
    \item \label{enum:an-relations-iv} (Row permutation and scaling) Let $N$ be a monomial matrix in $\SL_n$, and let $\sigma\in \fr{S}_n$ be the associated permutation. Then we have
    \[
        \langle M\rangle = (-1)^\sigma \langle NM\rangle.
    \]
\end{enumerate}
\end{lemma}
\begin{proof}
    Relations \eqref{enum:an-relations-i} and \eqref{enum:an-relations-ii} are rephrasings of relations \eqref{enum:steinberg-1}, \eqref{enum:steinberg-2}, and \eqref{enum:steinberg-3} of \cref{def:steinberg} in $\St_n$. To prove \eqref{enum:an-relations-iii}, we give the proof for $i=1$ and $j=2$; the other cases are similar. Using relation \eqref{enum:steinberg-1} in $\St_n$, we can write
    \[
        [e_1,\compactldots,e_n] = [-\lambda e_{1}+e_2,e_2,\compactldots,e_n]+[e_1,-\lambda e_1+e_2,\compactldots,e_n] = [-\lambda e_{1}+e_2,e_2,\compactldots,e_n]+e_{12}(-\lambda).
    \]
    Using relation \eqref{enum:steinberg-2} in $\St_n$, we can change the first term to
    \[
        [-\lambda e_{1}+e_2,e_2,\compactldots,e_n] =[e_1-\lambda^{-1}e_2,e_2,\compactldots,e_n]  = e_{21}(-\lambda^{-1}).
    \]
    This yields an equation $\langle M\rangle = [\rm{id}_n]\otimes[M] = [e_{21}(-\lambda^{-1})]\otimes [M]+[e_{12}(-\lambda)]\otimes [M]$ in $\scr{A}_n$. Finally, observe that $e_{12}(-\lambda)^{-1} = e_{12}(\lambda)$ and $e_{21}(-\lambda^{-1})^{-1} = e_{21}(\lambda^{-1})$ and so in $\scr{A}_n$ we have 
    \[
        [e_{21}(-\lambda^{-1})]\otimes [M]+[e_{12}(-\lambda)]\otimes [M] = \langle e_{21}(\lambda^{-1})M\rangle+\langle e_{12}(\lambda)M\rangle
    \]
    which establishes \eqref{enum:an-relations-iii}. For \eqref{enum:an-relations-iv}, let $P$ denote the permutation matrix obtained by replacing each nonzero entry in $N^{-1}$ with $1$, so that we have $\det(P) =  (-1)^{\sigma}$. Since any permutation matrix can be obtained from column operations on the identity, we have an equality $[\rm{id}_n] = (-1)^\sigma[P] =  (-1)^{\sigma}[N^{-1}]$ in $\St_n$, where the second equality uses column scaling.  Since $N^{-1}\in \SL_n$, we have a string of equalities in $\scr{A}_n$ establishing \eqref{enum:an-relations-iv}:
    \[
        \langle M\rangle =[\rm{id}_n]\otimes [M] =  (-1)^{\sigma}[N^{-1}]\otimes [M] =  (-1)^{\sigma}[\rm{id}_n]\otimes [(NM)] =   (-1)^{\sigma}\langle NM\rangle. \qedhere
    \]
\end{proof}

\begin{remark}\label{rem:row-scaling}
We offer a quick note on the process of scaling rows using relation \eqref{enum:an-relations-iv} in \cref{lem:an-relations}. Since the monomial matrix $N$ is required to be in $\SL_n$, we are \emph{not} allowed to scale a single row by $\lambda\in F^{\times}\setminus \{1\}$.  On the other hand, we can scale one row by $\lambda$ at the cost of scaling another row by $\lambda^{-1}$. Row permutation using the same relation involves a similar correction: a permutation matrix which switches two rows is not in $\SL_n$, but a matrix which switches two rows and multiplies a (possibly different) row by $-1$ is in $\SL_n$. Thus, we may swap rows at the cost of scaling a third row by $-1$, and adding $-1$ on the outside. 
\end{remark}

Using these relations, we are now able to construct a smaller generating set of $\scr{A}_n$ by adapting the arguments in \cite[Section 2]{GKRW25}. Recall from \cref{nota:diag-sdiag} that $\mathrm{sdiag}(a_1,\compactldots,a_{n-1})$ denotes the superdiagonal matrix with entries $a_1,\dots,a_{n-1}\in F$ just above the 1's on the diagonal.

\begin{proposition}\label[proposition]{prop: Jordan reduction}
    $\scr{A}_n$ is generated by the elements $\langle \mathrm{sdiag}(a_1,\compactldots,a_{n-1})\rangle$ for $a_1,\dots,a_{n-1}\in F$.
\end{proposition}

\begin{proof}This can be verified by implementing the cyclicity part of the proof of \cite[Theorem C]{GKRW25} using \cref{lem:an-relations} but recalling that we cannot arbitrarily scale rows (see \cref{rem:row-scaling}). We will give the details below.

\smallskip

    Let $M = [v_1,\compactldots,v_n]$. We may use the column scaling relation of \cref{lem:an-relations} \eqref{enum:an-relations-ii} to assume that the last coordinate of each $v_i$ is either $0$ or $1$.  Since the set $\{v_i\}$ is a basis, at least one of them has a last coordinate which is $1$. If more than one column has a $1$ in the last coordinate we can use the column addition relation of \cref{lem:an-relations} \eqref{enum:an-relations-i} to write $\langle M \rangle$ as a sum $\langle M_1\rangle+\langle M_2\rangle$ where each of $M_1$ and $M_2$ has strictly fewer $1$'s in the last row.  Repeating this, we may assume that $M$ has a single $1$ in the last row.  By column swapping as in \cref{lem:an-relations} \eqref{enum:an-relations-ii}, we may assume that $M_{nn}=1$.

    Next, we would like to reduce to the case where the last column of $M$ is $0$, except for $M_{n,n}=1$ and $M_{n-1,n}$ is arbitrary. To that end, use the row scaling relation of \cref{lem:an-relations} \eqref{enum:an-relations-iv} to scale the first $n-1$ rows so that their last coordinates are all $0$ or $1$; this can be done at the cost of changing $M_{n,n}$ to some arbitrary nonzero element. Note that every element in the last row, aside from $M_{n,n}$, remains $0$. Now, using the row addition relation of \cref{lem:an-relations} \eqref{enum:an-relations-iii} on the first $n-1$ rows, possibly at the cost of changing the sign of $M_{n,n}$, we may write $\langle M\rangle$ as the sum of matrices $\langle N\rangle$ where the last row of $N$ is zero except for $N_{n,n}$ and the last column is zero except for $N_{n,n}$ and $N_{n-1,n}=1$. Scaling the last column by $N_{n,n}^{-1}$, we have that $\langle M\rangle$ is the sum of matrices of the form
    \[
        \left[\begin{array}{@{}c|c@{}}
  M'
  &  
  \begin{matrix}
  0  \\
  0 \\
  \vdots\\
  0\\
  * \\
  \end{matrix}
  \\
\hline
 \begin{matrix}
   0 & 0 & \dots & 0
  \end{matrix}
 &
  1
\end{array}\right]
    \]
    where $M'$ is in $\GL_{n-1}$ and $*$ is an arbitrary element of $F$. If one repeats the process above, working now on the first $n-1$ columns and rows of the matrix, we see that this matrix can be written as the sum of matrices of the form
    \[
       \left[\begin{array}{@{}c|c@{}}
  M''
  &  
  \begin{matrix}
  0 & 0  \\
  0 & 0 \\
  \vdots & \vdots\\
   * & 0\\
  \end{matrix}
  \\
\hline
 \begin{matrix}
   0 & 0 & \dots & 0\\
   0 & 0 & \dots & 0
  \end{matrix}
 &
  \begin{matrix}
     1 & * \\
      0 & 1
  \end{matrix}
\end{array}\right]
    \]
    where $M''$ is in $\GL_{n-2}$.  Repeating this process $n$ times proves the claim.
\end{proof}

Finally, we observe that some of these generators can be simplified further:

\begin{proposition}\label{prop: superdiagonals}
    If $a_1,\dots,a_{n-1}\in F$ are all nonzero then we have
    \[
        \langle \mathrm{sdiag}(a_1,\compactldots,a_{n-1})\rangle = \left\{ {\textstyle \prod}_{i=1}^{n-1} a_i^{n-i}\right\}_n.
    \]
\end{proposition}

\begin{example}Before giving the general proof, we give an example with $n=4$ which demonstrates the method; this example is essentially generic. First, we have 
\[
    \begin{bmatrix}
        1 & a_1 & 0 & 0\\
        0 & 1 & a_2 & 0\\
        0 & 0 & 1   & a_3\\
        0 & 0 & 0   & 1\\
    \end{bmatrix}
    \mapsto 
    \begin{bmatrix}
        1 & 1 & 0 & 0\\
        0 & a_1^{-1} & a_2 & 0\\
        0 & 0 & 1   & a_3\\
        0 & 0 & 0   & 1\\
    \end{bmatrix}
    \mapsto\begin{bmatrix}
        1 & 1 & 0 & 0\\
        0 & 1 & a_1a_2 & 0\\
        0 & 0 & 1   & a_3\\
        0 & 0 & 0   & a_1^{-1}\\
    \end{bmatrix}
\]
where the first operation is column scaling and the second is a row scaling.  We then repeat:
\[
\begin{bmatrix}
        1 & 1 & 0 & 0\\
        0 & 1 & a_1a_2 & 0\\
        0 & 0 & 1   & a_3\\
        0 & 0 & 0   & a_1^{-1}\\
    \end{bmatrix}
    \mapsto 
    \begin{bmatrix}
        1 & 1 & 0 & 0\\
        0 & 1 & 1 & 0\\
        0 & 0 & (a_1a_2)^{-1}   & a_3\\
        0 & 0 & 0   & a_1^{-1}\\
    \end{bmatrix}
    \mapsto
    \begin{bmatrix}
        1 & 1 & 0 & 0\\
        0 & 1 & 1 & 0\\
        0 & 0 & 1   & a_1a_2a_3\\
        0 & 0 & 0   & a_1^{-2}a_2^{-1}\\
    \end{bmatrix}
\]
and finally we scale the last column by $a_1^2a_2$, yielding $\{a_1^3a_2^2a_3\}_4$, as claimed.\end{example}

\begin{proof}[Proof of \cref{prop: superdiagonals}]
    We show, for any $1\leq k\leq n-2$, that
    \[
    \langle \mathrm{sdiag}(1,\compactldots,1,a_{k},a_{k+1},\dots,a_{n-1})\rangle =\langle \mathrm{sdiag}(1,\dots,1,a_ka_{k+1},a_{k+2},a_{k+3},\compactldots,a_ka_{n-1})\rangle
    \]
    where the term on the right has a $1$ in the $k$th spot and multiplies the $(k+1)$th and last coordinates by $a_k$ (if the $k+1$th is the last coordinate, we multiply it by $a_k^2$).  Applying this for $k=1,\dots,n-2$ yields the claim. By definition we have
    \[
    \langle \mathrm{sdiag}(1,\compactldots,1,a_{k},a_{k+1},\compactldots,a_{n-1})\rangle = \langle e_1,\dots,e_{k-1}+e_{k},a_ke_k+e_{k+1},\dots,a_{n-1}e_{n-1}+e_n \rangle
    \]
    where $a_ke_k+e_{k+1}$ occurs in the $(k+1)$th column. Scaling the $(k+1)$th column by $a_k^{-1}$, we see this equals
    \[
        \langle e_1,\compactldots,e_{k-1}+e_{k},e_k+a_k^{-1}e_{k+1},\compactldots,a_{n-1}e_{n-1}+e_n\rangle .
    \]
    Next, scaling the $(k+1)$th row by $a_k$, and the last row by $a_k^{-1}$, we see this equals
    \[
        \langle e_1,\compactldots,e_{k-1}+e_{k},e_k+e_{k+1},a_ka_{k+1}e_{k+1}+e_{k+2},\compactldots,a_{n-1}e_{n-1}+a_k^{-1}e_n \rangle.
    \]
    Finally, scaling the last column by $a_k$, this is equal to 
    \[
        \langle e_1,\dots,e_{k-1}+e_{k},e_k+e_{k+1},a_ka_{k+1}e_{k+1}+e_{k+2},\dots,a_ka_{n-1}e_{n-1}+e_n \rangle
    \]
    which is $\langle \mathrm{sdiag}(1,\dots,1,a_ka_{k+1},a_{k+2},a_{k+3},\dots,a_ka_{n-1})\rangle$ as claimed.
\end{proof}

\subsection{Generating $\scr{A}_*$ as a graded $\bb{Z}[F^\times]$-algebra} To construct an even smaller generating set, we show that the abelian groups $\scr{A}_n$ assemble into a graded $\bb{Z}[F^\times]$-algebra.

We start by defining a $\bb{Z}[F^\times]$-module structure on $\scr{A}_n$. Since $\St^2_n$ has an action not just by the group $\SL_n$ but by the larger group $\GL_n$, the coinvariants $\scr{A}_n$ admit a residual action of $\GL_n/\SL_n\cong F^{\times}$.  Explicitly, if $\lambda\in F^{\times}$ and $I^{\lambda}\in \GL_n$ denotes the diagonal matrix $I^{\lambda} = \rm{diag}(1,\compactldots,1, \lambda)$, then we have $\lambda \cdot [M]\otimes [N] = [I^{\lambda}M]\otimes [I^{\lambda}N]$. Since $[I^{\lambda}] = [\rm{id}_n]$ in $\St_n$, it follows that $\lambda\cdot \langle M\rangle = \langle I^{\lambda}M\rangle$ for any $M\in \GL_n$.

\begin{lemma}\label{lem:f-times-action} \,
    \begin{enumerate}[(i)]
    \item \label{enum:f-times-action-i}For any $a,\lambda\in F^{\times}$ we have
    $\lambda\cdot \{a\}_n = \{\lambda^{-1}a\}_n$.
    \item \label{enum:f-times-action-ii} For any $\lambda\in F^{\times}$ we have $\{\lambda^n\}_n = \{1\}_n$.
    \end{enumerate}
\end{lemma}
\begin{proof}
    By the comment preceding the statement of this lemma, we have that $\lambda \cdot \{1\}_n = \langle M\rangle$ where $M$ is the matrix which is equal to the matrix representing $\{1\}_n$, but with $M_{nn}=\lambda$ instead of $1$.  Scaling the last column of $M$ by $\lambda^{-1}$, we see that $\lambda\cdot \{1\}_n=\langle M\rangle = \{\lambda^{-1}\}_n$ for any $\lambda\in F^{\times}$.  Thus, for any $\lambda,a\in F^{\times}$, the following equation proves \eqref{enum:f-times-action-i}
    \[
        \lambda\cdot\{a\}_n = \lambda\cdot a^{-1}\cdot\{1\}_n = \{\lambda^{-1}a\}_n.
    \]
    To prove \eqref{enum:f-times-action-ii}, we begin by scaling the first row of $\{1\}_n$ by $\lambda$ and the last row by $\lambda^{-1}$ so we have 
    \[
        \{1\}_n = \langle [\lambda e_1,\lambda e_1+e_2, e_2+e_3,\compactldots,e_{n-1}+\lambda^{-1}e_{n}]\rangle.
    \]
    Scaling the first column by $\lambda^{-1}$ and the last column by $\lambda$ we have 
    \[
        \langle [\lambda e_1,\lambda e_1+e_2, e_2+e_3,\compactldots,e_{n-1}+\lambda^{-1}e_{n]}]\rangle = \langle\mathrm{sdiag}(\lambda,1,1\compactldots,1,\lambda)\rangle
    \]
    which is equal to $\{\lambda^{n-1}\lambda\}_n = \{\lambda^{n}\}_n$ by \cref{prop: superdiagonals}.
\end{proof}

\begin{remark}
    \cref{lem:f-times-action} \eqref{enum:f-times-action-ii} also follows from the observation that the residual action of the scalar matrix $\lambda \cdot \id_n\in Z(\GL_n)$ sends $\{a\}_n$ to $\{\lambda^na\}_n$. But this matrix centralises $\SL_n,$ so we must have  $\{a\}_n=\{\lambda^na\}_n$.
\end{remark}
Next we describe an associative nonunital product on $\scr{A}_* \coloneq \bigoplus_{n \geq 1} \scr{A}_n$. There is a natural pairing
\[
    \scr{A}_n\otimes\scr{A}_m \longrightarrow \scr{A}_{n+m}
\]
defined on elements of the form $\gen{M}$ via block sum of matrices: it sends $\gen{M}\otimes\gen{N}$ to $\gen{M\oplus N}$. This is compatible with the $\bb{Z}[F^{\times}]$-module structure, making $\scr{A}_*$ into a graded associative nonunital $\bb{Z}[F^{\times}]$-algebra. 

\begin{remark}
    The product structure on $\scr{A}_*$ defined above agrees with the product structure mentioned in the introduction. Upon interpreting $\scr{A}_*$ in terms of the $E_2$-homology groups of $\mathbf{BSL}(F)_\bb{Z}$ (see \cref{sec:ek-cells}) it is the residual product obtained from the $E_\infty$-algebra structure.
\end{remark}

\begin{theorem}\label{theorem: An generated by 1s}
    The graded $\bb{Z}[F^{\times}]$-algebra $\scr{A}_*$ is generated by classes of the form $\{1\}_n$ for $n\geq 1$. In fact, for every fixed $n$ the $\bb{Z}[F^{\times}]$-module $\scr{A}_n$ is generated by $\{1\}_n$.
\end{theorem}

\begin{proof}
    We will prove the second statement of the theorem (as it certainly implies the first): By \cref{prop: Jordan reduction}, it suffices to prove that every element of the form $\langle \mathrm{sdiag}(a_1,\compactldots,a_{n-1})\rangle$ is a $\bb{Z}[F^{\times}]$-multiple of $\{1\}_n$. Every such element $\langle \mathrm{sdiag}(a_1,\compactldots,a_{n-1})\rangle$ is by construction a product $\langle J_1\rangle \cdots \langle J_r\rangle$ of $r$-many superdiagonal matrices $J_i$ with no zeros on the superdiagonal, where $r$ is one more than the number of $i$ for which $a_i=0$. By \cref{prop: superdiagonals}, each $\langle J_i\rangle$ can be expressed as an element of the form $\{\lambda_i\}_{b_i}$ for some $\lambda_i \in F^{\times}$. Thus, it suffices to establish the following claim, which is the content of \cref{sec:proof-of-claim} below.
    \begin{claim}
        \label{claim:product} Every product $\{\lambda_1\}_{a}\cdot \{\lambda_2\}_{b}$ can be expressed as a sum of terms of the form $\{\mu_j \}_{a+b}$.
    \end{claim}
    \noindent Indeed, using that $\mu^{-1}\cdot \{1\}_n = \{\mu\}_n$ by \cref{lem:f-times-action} \eqref{enum:f-times-action-i}, the following string of equalities then completes the proof:
    \[
       \langle \mathrm{sdiag}(a_1,\compactldots,a_{n-1})\rangle = \langle J_1\rangle \cdots \langle J_r\rangle = \{\lambda_1\}_{b_1} \cdots \{\lambda_r\}_{b_r} = \sum_j \{\mu_j\}_{n} = \left( \sum_j \mu_j^{-1} \right) \cdot \{1\}_{n}.\qedhere
    \]
\end{proof}

\subsection{Proof of \cref{claim:product}}
\label{sec:proof-of-claim}
The missing step in the argument for \cref{theorem: An generated by 1s} is the proof of \cref{claim:product}, which occupies the remainder of this section. We start by observing that
\[
    \{\lambda_1\}_a\cdot \{\lambda_2\}_b = (\lambda_1\lambda_2)^{-1}\cdot \{1\}_a\cdot \{1\}_b.
\]
Thus, it suffices to prove that each product $\{1\}_a\cdot \{1\}_b$ is the sum of terms of the form $\{\mu_i\}_{a+b}$. Our formula for these products features the following two sequences of numbers:

\begin{definition}\label{defn: lozanic}
    We recursively define numbers $P(n,k)$ and $Q(n,k)$ for $n\geq k\geq 0$.  We let $P(n,k)=1$ if $n\leq 2$ or $k\in \{0,n\}$.  From there, $P(n,k)$ is defined recursively as
\[
    P(n,k) \coloneq \begin{cases}
        P(n-1,k-1)+P(n-1,k)- {{(n/2)-1}\choose{(k-1)/2}} & \text{if $n$ is even and $k$ is odd,}\\
        P(n-1,k-1)+P(n-1,k) & \text{else.}
    \end{cases}
\]
In all cases we define $Q(n,k) \coloneqq {{n}\choose{k}}-P(n,k)$.
\end{definition}

\begin{theorem}\label{thm: products}
    For $a,b \geq 1$ we have 
    \[
        \{1\}_a\cdot \{1\}_b = P(a+b,a)\{1\}_{a+b}+Q(a+b,a)\{-1\}_{a+b}.
    \]
\end{theorem}

\begin{example}
    In the quotient $(\scr{A}_n)_{F^{\times}}$, we have $\ol{\{1\}}_n = \ol{\{-1\}}_n$ and this product formula becomes the divided power structure that appears in \cite[Theorem 6.9]{GKRW25}
    \[
        \ol{\{1\}}_a\cdot \ol{\{1\}}_b = P(a+b,a)\ol{\{1\}}_{a+b}+Q(a+b,a)\ol{\{1\}}_{a+b} = {\textstyle {{a+b}\choose {a}}} \ol{\{1\}}_{a+b}.
    \]
    By construction the quotient map $\bigoplus_{n \geq 1} (\St^2_n)_{\SL_n} \to \bigoplus_{n \geq 1} (\St^2_n)_{\GL_n}$ is a map of algebras.
\end{example}

\begin{example}\label{exam:products-odd}
    When $a+b$ is odd we have that $\{1\}_{a+b} = \{-1\}_{a+b}$, which can be seen by scaling the first $(a+b)-1$ rows by $-1$ and then scaling the first $(a+b)-1$ columns by $-1$ using \cref{lem:an-relations}. This works more generally for any $a,b$ when $F$ contains an $n$th root of $-1$ (or equivalently, a $\nu_2(n)$th root of $-1$). Thus, in this case the formula of \cref{thm: products} simplifies to
    \[
        \{1\}_a\cdot \{1\}_b = (P(a+b,a)+Q(a+b,a))\{1\}_{a+b} = {\textstyle {{a+b}\choose{a}}}\{1\}_{a+b}.
    \]
\end{example}

Before commencing with the proof of \cref{thm: products}, we observe that the numbers $P(n,k)$ and $Q(n,k)$ also admit a closed form in terms of binomial coefficients. This is the content of the next lemma, whose proof is routine and hence omitted.

\begin{lemma}\label{lem:expressions-of-p-and-q}
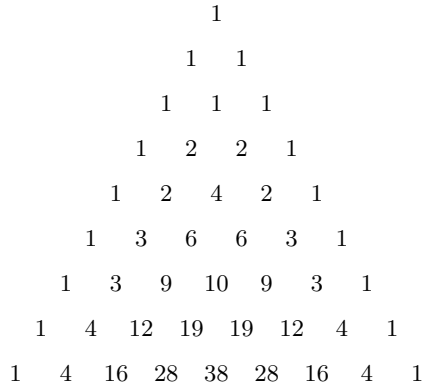
    Let $0<k<n$, and write $n_2 \coloneq \lfloor n/2\rfloor$ and $k_2 \coloneq \lfloor k/2\rfloor$. We have
    \begin{align*}
        P(n,k) &= \begin{cases}
            \frac{1}{2}{n\choose{k}} & \text{if $n$ is even and $k$ is odd,}\\
            \frac{1}{2}\left({n\choose{k}}+ {n_2\choose{k_2}}\right) & \text{otherwise,}
        \end{cases} \\
    Q(n,k) &= \begin{cases}
            \frac{1}{2}{n\choose{k}} & \text{if $n$ is even and $k$ is odd.}\\
            \frac{1}{2}\left( {n\choose{k}}-{n_2\choose{k_2}}\right) & \text{otherwise.}
        \end{cases}
    \end{align*}
\end{lemma}

\begin{remark}
    The numbers $P(n,k)$ form a number triangle called \emph{Losanitsch's triangle} (see \cref{fig:triangle} and \cite{oeis}) and were first studied in relation to the symmetries of alkanes \cite{Losanitsh}. This is a deformation of Pascal's triangle, with the difference between the two measured by the terms $Q(n,k)$. 
\end{remark}

\begin{figure}
\centering \begin{tikzpicture}[
    x=0.95cm, y=0.85cm,
    every node/.style={inner sep=0pt, font=\small},scale=.7]
  \foreach \row [count=\n from 0] in {
      {1},
      {1,1},
      {1,1,1},
      {1,2,2,1},
      {1,2,4,2,1},
      {1,3,6,6,3,1},
      {1,3,9,10,9,3,1},
      {1,4,12,19,19,12,4,1},
      {1,4,16,28,38,28,16,4,1}}
    \foreach \val [count=\k from 0] in \row
      \node at (\k-\n/2, -\n) {\val};
\end{tikzpicture}
\caption{Losanitsch's triangle, with $i$th entry (starting at $0$) on the $j$th row (starting at $0$) given by $P(j-i,i)$.}
\label{fig:triangle}
\end{figure}

The next lemmas are elementary consequences of Pascal's identity:

\begin{lemma}\label{lemma: P equal Q}
    If $n$ is even and $k$ is odd then $P(n,k) = Q(n,k) = \frac{1}{2}{n\choose k}$.
\end{lemma}

\begin{lemma}\label{lemma: P plus Q even n}
    If $n$ is even and $k$ is odd then $$P(n-1,k-1)+Q(n-1,k) = P(n,k)=Q(n,k)= Q(n-1,k-1)+P(n-1,k).$$
\end{lemma}

\begin{lemma}\label{lemma: P plus Q odd n}
    If $n$ is odd and $k$ is even then $$P(n-1,k-1)+Q(n-1,k) = Q(n,k) \text{ and } Q(n-1,k-1)+P(n-1,k)=P(n,k).$$
\end{lemma}

We now begin working towards \cref{thm: products}. We will only give a proof for the case where $a+b$ is even; the case where $n=a+b$ is odd is largely the same but easier (see \cref{exam:products-odd}).  We thus fix an \emph{even} integer $n\geq 2$ and we will always have $1\leq a\leq n-1$, $b=n-a$, and $1\leq i\leq a$. The first step in the proof of \cref{thm: products} is to decompose the product $\{1\}_a\{1\}_b$ as a sum of simpler matrices $M(n,k,i)$ and $N(n,k,i)$, which we define now.

\begin{notation}For $1 \leq a \leq n-1$ we write $j_{a+1} = e_a+e_{a+1}$.\end{notation}

\begin{definition}\,
\begin{enumerate}[\noindent (i)]
    \item For $1\leq i\leq a< n$ we write $M(n,a,i)$ for the class of the matrix \[[e_1,j_2,\compactldots,j_a,e_{i}+e_{a+1},j_{a+2},j_{a+3},\compactldots,j_{n}].\]  That is, $M(n,a,i)$ is the same as $\{1\}_n$ except that the $(a+1)$th column is $e_{i}+e_{a+1}$ instead of $j_{a+1} = e_{a}+e_{a+1}$. 
    \item  For $1\leq i\leq a< n$ we write $N(n,a,i)$ for the same matrix as $M(n,a,i)$ except that the last column is $e_{n-1}-e_n$ instead of $j_n=e_{n-1}+e_{n}$.  That is, it is obtained from $M(n,a,i)$ by replacing the $1$ in the bottom right of the matrix by $-1$.
\end{enumerate}
\end{definition}

\begin{example}
    \[
        M(4,2,1) = \begin{bmatrix}
            1 & 1 & 1 & 0\\
            0 & 1 & 0 & 0\\
            0 & 0 & 1 & 1\\
            0 & 0 & 0 & 1\\
        \end{bmatrix}
        \quad \mathrm{and}\quad
        N(4,3,2) = \begin{bmatrix}
            1 & 1 & 0 & 0\\
            0 & 1 & 1 & 1\\
            0 & 0 & 1 & 0\\
            0 & 0 & 0 & -1\\
        \end{bmatrix}.
    \]
\end{example}

When $n$ is fixed and clear from context we will omit it from the notation and write $M(a,i) = M(n,a,i)$, and similarly for $N(a,i)$. By definition we have $M(a,a) = \{1\}_n$.  By scaling the last column by $-1$ we get $N(a,a) = \{-1\}_n$.  The product $\{1\}_a\{1\}_b$ can be expressed as a sum of $M(a,i)$ and $N(a,i)$:

\begin{proposition}\label{proposition: products in terms of M and N}
    We have \[\{1\}_a\cdot\{1\}_b=\begin{cases}
        M(a,1)+M(a+1,1) & \text{if $a$ is even,}\\
        M(a,1)+N(a+1,1) & \text{if $a$ is odd.}
    \end{cases}\]
\end{proposition}
\begin{proof}
    By definition, writing $a+b=n$, we have
    \[
        \{1\}_a\cdot\{1\}_b = [e_1,j_2,\compactldots,j_a,e_{a+1},j_{a+2},\compactldots,j_n].
    \] We start by adding the first column to the $(a+1)$-st column which gives
    \[
        \{1\}_a\cdot\{1\}_b =M(a,1)+[e_{1}+e_{a+1},j_2,\compactldots,j_a,e_{a+1},j_{a+2},\compactldots,j_n]
    \]
    using  \cref{lem:an-relations} \eqref{enum:an-relations-i}. It remains to show that the second term is equal to $M(a+1,1)$ if $a$ is even and equal to $N(a+1,1)$ if $a$ is odd.  We simplify this term by first permuting the columns according to the cycle $(1,2,\dots,a+1)$ which gives
    \[
        [e_{1}+e_{a+1},j_2,\compactldots,j_a,e_{a+1},j_{a+2},\compactldots,j_n] = (-1)^a[e_{a+1},e_1+e_{a+1},j_2,\compactldots,j_{a},j_{a+2},\compactldots,j_n].
    \]
    Next, we cycle the rows according to the same cycle, at the cost of multiplying on the outside, and in the last row, by $(-1)^a$ (see \cref{rem:row-scaling}). This gives
    \[
        [e_{1}+e_{a+1},j_2,\compactldots,j_a,e_{a+1},j_{a+2},\compactldots,j_n] = [e_1,j_2,\compactldots,j_{a+1},e_{1}+e_{a+2} ,\compactldots,j_{n-1},e_{n-1}+(-1)^ae_n],
    \]
    implying the result.
\end{proof}

The plan for computing the product $\{1\}_a\cdot\{1\}_b$ is to show that the terms $M(a,i)$ and $N(a,i)$ satisfy a recursion formula which allows us to compute them in terms of $\{1\}_{a+b}$ and $\{-1\}_{a+b}$. We will establish several lemmas making this recursion precise; their proofs are similar to that of \cref{proposition: products in terms of M and N}.

\begin{lemma}\label{lemma: M and N reduction in case a=n-1}
    For $n$ even and any $i<n-1$ there are equalities
    \[
        M(n-1,i) = N(n-1,i+1)+\{1\}_n \quad \text{and} \quad 
        N(n-1,i) = M(n-1,i+1)+\{-1\}_n.
    \]
\end{lemma}
\begin{proof}
    We give the proof for $M(n-1,i)$, as that for $N(n-1,i)$ is essentially the same.  By definition we have
    \[
        M(n-1,i) = [e_1,j_2,\compactldots,j_{n-1},e_i+e_n]
    \]
    and we apply the column addition rule, subtracting the $n$th column from the $(i+1)$th column giving
    \begin{align*}
        M(n-1,i) &  = [e_1,j_2,\compactldots,e_{i+1}-e_n]+[e_1,\compactldots,j_{i},e_{i+1}-e_n,j_{i+2},\compactldots,e_i+e_n]\\
        &  = N(n-1,i+1)+[e_1,\compactldots,j_{i},e_{i+1}-e_n,j_{i+2},\compactldots,e_i+e_n].
    \end{align*}
    It remains to check that the second term is equal to $\{1\}_n$.

    To do so, cycle the columns of the second term by the permutation $(i+1,\dots, n)$ to get
    \[
        [e_1,\compactldots,j_{i},e_{i+1}-e_n,j_{i+2},\compactldots,e_i+e_n] = (-1)^{n-i-1} [e_1,\dots,j_{i},e_i+e_n,e_{i+1}-e_n,j_{i+2},\compactldots,j_{n-1}].
    \]
    We then cycle the rows of the matrix on the right hand side according to $(i+1,\dots,n)$, giving 
    \[
         (-1)^{n-i-1} [e_1,\dots,j_{i},e_i+e_n,e_{i+1}-e_n,j_{i+2},\compactldots,j_{n-1}]  =  [e_1,\compactldots,j_{i},j_{i+1},e_{i+2}-e_{i+1},j_{i+3},\compactldots,j_{n-1},v_n]
    \]
    where $v_n = e_{n-1}+(-1)^{n-i-1}e_n$. After scaling the last column by $(-1)^{n-i-1}$ this matrix is given by $\mathrm{sdiag}(1,1,\compactldots,1,-1,1,\compactldots,(-1)^{n-i-1})$ where the first $-1$ appears in the $(i+1)$th spot.  Note that here we have implicitly assumed that $i<n-2$ to preclude the possibility that the two negative ones appear in the same column. The result remains true if $i = n-2$,  but the argument needs to be modified slightly; we leave the details to the reader.  By \cref{prop: superdiagonals}, this result is equal to $\{(-1)^{n-i-1}\cdot(-1)^{n-i-1}\}_n = \{1\}_n$.
\end{proof}

The following proposition is now a straightforward induction argument.
\begin{proposition}\label{prop: the case a is n minus 1}
    For $n$ even and $i\leq n-1$, we have
    \begin{align*}
        M(n-1,i) &= \begin{cases}
                    \left(\frac{n+1-i}{2}\right)\{1\}_n+\left(\frac{n-1-i}{2}\right)\{-1\}_n & \mathrm{for }\ n-i\ \mathrm{odd},\\
                    \left(\frac{n-i}{2}\right)(\{1\}_n+\{-1\}_n)& \mathrm{for }\ n-i\ \mathrm{even},
                    \end{cases} \\
    N(n-1,i) &= \begin{cases}
                    \left(\frac{n-1-i}{2}\right)\{1\}_n+\left(\frac{n+1-i}{2}\right)\{-1\}_n & \mathrm{for }\ n-i\ \mathrm{odd},\\
                    \left(\frac{n-i}{2}\right)(\{1\}_n+\{-1\}_n)& \mathrm{for }\ n-i\ \mathrm{even}.
                    \end{cases}
    \end{align*}
\end{proposition}

\begin{lemma}
    For $1\leq i\leq a< n$, we let $$L(n,a,i) = L(a,i) \coloneq [e_1,j_2,\compactldots,j_{a},-e_i+e_{a+1},j_{a+2},\compactldots,j_n].$$ That is, $L(a,i)$ is like $M(a,i)$  but with $-1$ instead of $1$ in the exceptional position $(i,a+1)$. Then $L(a,i) = M(a,i)$ for $n-a$ even and $L(a,i)=N(a,i)$ for $n-a$ odd. 
\end{lemma}
\begin{proof}
    The proof follows from the fact that we can scale rows by $-1$ but only if we do so two at a time.  If $n-a$ is even then by scaling the last $n-a$ rows we have
    \[
     L(a,i) = [e_1,\compactldots,j_a,-e_i-e_{a+1},-j_{a+2},\compactldots,-j_n]
    \]
    and we then scale the last $n-a$ columns by $-1$ to obtain $L(a,i)=M(a,i)$.  If $n-a$ is odd we scale the rows $a+1$ through $n-1$ by $-1$ and obtain
    \[
        L(a,i) = [e_1,\compactldots,j_a,-e_i-e_{a+1},-j_{a+2},\compactldots,-e_{n-1}+e_n]
    \]
    and we scale the last $n-a$ columns by $-1$ to obtain $L(a,i) = N(a,i)$.
\end{proof}

We need one final lemma.
\begin{lemma}\label{lemma: M and N reduction}
    For $1<i<a<n-1$ and $n$ even, we have
    \begin{align*}
        M(a,i) &= \begin{cases}
            M(a,i+1)+M(a+1,i+1) & \text{for $a$ even}, \\
            N(a,i+1)+M(a+1,i+1) & \text{for $a$ odd},
        \end{cases} \\   
        N(a,i) &= \begin{cases}
            N(a,i+1)+N(a+1,i+1) & \text{for $a$ even,} \\
            M(a,i+1)+N(a+1,i+1) & \text{for $a$ odd.}
        \end{cases}
    \end{align*}
\end{lemma}
\begin{proof}
    We start with the case of $M(a,i)$, and recall that $M(a,i) = [e_1,j_2,\compactldots,j_{a},e_i+e_{a+1},j_{a+2},\compactldots,j_n]$. We first subtract the $(a+1)$th column from the $(i+1)$th column to obtain
    \begin{align*}
        M(a,i) =& [e_1,j_2,\compactldots,j_{a},e_{i+1}-e_{a+1},j_{a+2},\compactldots,j_n] \\
        &+ [e_1,j_2,\compactldots, j_{i},e_{i+1}-e_{a+1},\compactldots,j_{a},e_i+e_{a+1},j_{a+2},\compactldots,j_n]\end{align*}
    and the first term is, after scaling the $(a+1)$th column by $-1$, equal to $L(a,i+1)$. By the previous lemma (using that the parity of $n-a$ agrees with that of $a$ since $n$ is even), it suffices to show that the second term is equal to $M(a+1,i+1)$. For brevity, we call this term $R$.

    Cycling the columns according to the permutation $(i+1,i+2,\dots a+1)$ we have
    \[
        R = (-1)^{a-i}[e_1,j_2,\compactldots, j_{i},e_i+e_{a+1}, e_{i+1}-e_{a+1},\compactldots,j_{a},j_{a+2},\compactldots,j_n].   
        \]
    We then cycle the rows according to the same permutation which eliminates the $(-1)^{a-i}$ on the outside but necessitates multiplying the last row by $(-1)^{a-i}$.  Thus we have
    \[
        R = [e_1,j_2,\compactldots, j_{i},j_{i+1}, e_{i+2}-e_{i+1},\compactldots,j_{a+1},e_{i+1}+e_{a+2},\compactldots,j_{n-1},e_{n-1}+(-1)^{a-i}e_n]
    \]
    which is exactly $N(a+1,i+1)$, except there is a $-1$ on the superdiagonal in matrix position $(i+1,i+2)$.  As in the last lemma, we may scale rows $i+2$ to $a+1$ by $-1$ and multiply the last row by $(-1)^{a-i}$ and then scale columns $i+2$ to $a+1$ by $-1$ to show that $R = M(a+1, i+1)$.
    
    The computation of $N(a,i)$ is essentially the same except that the bottom right entry in the matrices has an extra factor of $-1$ throughout.  
\end{proof}

\begin{proposition}\label{proposition: M plus N}
    For $n$ even and any $a$ and $i$, we have
    \[
        M(a,i)+N(a,i) = {{n-i}\choose{a-i}}\cdot\left(\{1\}_n+\{-1\}_n\right).
    \]
\end{proposition}

\begin{proof}
    We work by downward induction on the values of $a$ and $i$.  The formulas for $M(n-1,i)$ and $N(n-1,i)$ in \cref{prop: the case a is n minus 1} establish the formula for the maximal choice $a=n-1$.  The maximal case $i = a$ is clear from the identities $M(a,a) = \{1\}_n$ and $N(a,a) = \{-1\}_n$.

    Fix a pair $(a,i)$ and suppose that we have established the formula for any pair $(b,j)$ with either $b>a$ or $j>i$.  Applying \cref{lemma: M and N reduction} we have
    \[
        M(a,i)+N(a,i) = M(a,i+1)+N(a,i+1)+M(a+1,i+1)+N(a+1,i+1).
    \]
    Note that this does not depend on whether $a$ is even or odd, even though the formulas in \cref{lemma: M and N reduction} do depend on this.  Applying the induction hypothesis, we have
    \[
        M(a,i)+N(a,i) = \left({{n-i-1}\choose{a-i-1}}+{{n-i-1}\choose{a-i}} \right)\cdot (\{1\}_n+\{-1\}_n)
    \]
    which is equal to ${{n-i}\choose{a-i}}\cdot (\{1\}_n+\{-1\}_n)$, by Pascal's identity.
\end{proof}

\begin{proposition}\label{proposition: M in terms of P and Q}
    For $n$ even and all $r\leq a< n$, we have 
    \begin{align*}M(n,a,r) &= P(n-r,a-r)\{1\}_n+Q(n-r,a-r)\{-1\}_n\\
    N(n,a,r) &= Q(n-r,a-r)\{1\}_n+P(n-r,a-r)\{-1\}_n.\end{align*}
\end{proposition}

\begin{proof}
    It suffices to establish the claim for $M(n,a,r)$ since the claim for $N(n,a,r)$ then follows immediately from \cref{proposition: M plus N}.  We define a partial ordering on triples $(n,a,r)$ where we have $(n,a,r)\geq (m,b,s)$ if
    \begin{itemize}
        \item $n>m$,
        \item $n=m$ and $a<b$,
        \item $n=m$, $a=b$, and $r<s$.
    \end{itemize}
    That is, this ordering is just lexicographical but with the ordering in the second and third positions being the opposite of the usual ordering.  This ordering is bounded below, and therefore we can prove the theorem by (strong) induction up this partial ordering.  Note that the triples $(n,a,r+1)$ and $(n,a+1,r+1)$ are both less than $(n,a,r)$ in this ordering.

    We first consider two base cases: If $r=a$, then \[M(n,a,a)=\{1\}_n = P(n-a,0)\{1\}_n+Q(n-a,0)\{-1\}_n\] because $P(n-a,0) = 1$ and $Q(n-a,0) = 0$ by \cref{defn: lozanic}. If $r < a=n-1$, then \[M(n,n-1,r)= P(n-r,n-r-1)\{1\}_n+Q(n-r,n-r-1)\{-1\}_n\]
    by \cref{prop: the case a is n minus 1} and \cref{lem:expressions-of-p-and-q}.

    Suppose then that we have proved the claim for all triples which are less than $(n,a,r)$ in the ordering.  If $r=a$ we are done, so assume that $r<a$.  We will apply \cref{lemma: M and N reduction}, but this requires separating into cases based on whether or not $a$ is even or odd.  Suppose first that $a$ is even.  We have 
    \begin{align*}
        M(n,a,r) & = M(n,a,r+1)+ M(n,a+1,r+1)\\
        &= (P(n-r-1,a-r-1)+P(n-r-1,a-r))\{1\}_n\\&+(Q(n-r-1,a-r-1)+Q(n-r-1,a-r))\{-1\}_n\\
        &  = P(n-r,a-r)\{1\}_n+Q(n-r,a-r)\{-1\}_n,
    \end{align*}
    where the second equality is the induction hypothesis and the third uses \cref{defn: lozanic}, the fact that $n-r$ and $a-r$ have the same parity and Pascal's identity.
    When $a$ is odd 
    \begin{align*}
        M(n,a,r) & = N(n,a,r+1) + M(n,a+1,r+1)\\
        &= (Q(n-r-1,a-r-1)+P(n-r-1,a-r))\{1\}_n\\&+(P(n-r-1,a-r-1)+Q(n-r-1,a-r))\{-1\}_n
    \end{align*}
    and we must separate into two cases based on whether $r$ is even or odd.  When $r$ is even the result follows from \cref{lemma: P plus Q even n} and when $r$ is odd it follows from \cref{lemma: P plus Q odd n}.  Thus we have established the induction step, proving the proposition.
\end{proof}

Finally, we are able to prove \cref{thm: products}.

\begin{proof}[Proof of \cref{thm: products}]
    Suppose that $n$ and $a$ are even. We combine \cref{proposition: M in terms of P and Q} with \cref{proposition: products in terms of M and N}, \cref{defn: lozanic} and Pascal's identity to get
\begin{align*}
    \{1\}_a\cdot\{1\}_b &= M(n,a,1)+M(n,a+1,1) \\
    & =(P(n-1,a-1)+P(n-1,a))\{1\}_n+(Q(n-1,a-1)+Q(n-1,a))\{-1\}_n\\
    & = P(n,a)\{1\}_n+Q(n,a)\{-1\}_n.
\end{align*}
\noindent Similarly, if $n$ is even and $a$ is odd we obtain
\begin{align*}
    \{1\}_a\cdot\{1\}_b &= M(n,a,1)+N(n,a+1,1) \\
    & =(P(n-1,a-1)+Q(n-1,a))\{1\}_n+(Q(n-1,a-1)+P(n-1,a))\{-1\}_n\\
    & = P(n,a)\{1\}_n+Q(n,a)\{-1\}_n
\end{align*}
where the last equality is \cref{lemma: P plus Q even n}. 
\end{proof}

\section{Theorem A: dimension 2}

The goal of this section is to prove Theorem \ref{mainthm:dimension 2}, calculating the double Steinberg coinvariants in rank $n=2$. Recall that the answer will be phrased in terms of the (symmetric) Grothendieck--Witt group $\rm{GW}(F)$, the group completion of the abelian monoid of isomorphism classes of symmetric bilinear forms over $F$ under orthogonal sum. This group is generated by the rank one forms $\langle a \rangle$ for $a \in F^\times$, subject to the relations (i) $\langle\lambda^2 a \rangle = \langle a \rangle$ for $a,\lambda \in F^\times$ and (ii) $\langle a \rangle + \langle b \rangle = \langle a+b\rangle + \langle ab(a+b) \rangle$ for $a,b,a+b \in F^\times$ (see e.g.\ \cite[Theorem 2.1.11]{deglise}). To obtain the Witt group $\rm{W}(F)$ as a quotient of $\rm{GW}(F)$, we add the further relation (iii) the hyperbolic form $h = \langle 1 \rangle + \langle -1 \rangle$ satisfies $h=0$. It is a well-known consequence of the Witt cancellation theorem that there is a pullback square of abelian groups (see e.g.~\cite[Theorem 2.1.11 (ii)]{Scharlau})
\[\begin{tikzcd} \rm{GW}(F) \rar{\rm{rk}} \dar[swap]{\rm{pr}} &[15pt] \bb{Z} \dar{\rm{pr}} \\[-5pt]
\rm{W}(F) \rar{\rm{rk}\,(\rm{mod}\,2)} & \bb{Z}/2\end{tikzcd}\]
where $\rm{rk}$ records the dimension of the underlying vector space. To construct the map $\St^2_{2}(F)_{\SL_2(F)}\to\rm{GW}(F)$ in Theorem \ref{mainthm:dimension 2}, it suffices to construct a pair of maps:
\begin{enumerate}[\noindent (1)]
    \item The first is the map to $\GL_n$-coinvariants
        \[\epsilon \colon \scr{A}_n \longrightarrow (\St^2_n)_{\GL_n} \overset{\cong}\longrightarrow \bbZ,\]
    where the right-hand computation is \cite[Theorem 2.4]{GKRW25}.
    \item The second is a map 
        \[\rm{scr} \colon \scr{A}_2 \longrightarrow \rm{W}(F),\]
    which takes the rank-one symmetric bilinear form with coefficient a reduced cross-ratio.
\end{enumerate}  
These maps will be compatible with the map to $\bb{Z}/2$, so they induce a map $\Phi\colon \St^2_{2}(F)_{\SL_2(F)}\to\rm{GW}(F)$ by the universal property of a pullback. We will prove $\Phi$ is an isomorphism by constructing an inverse $\Psi$ using the aforementioned presentation of $\rm{GW}(F)$.

\medskip

\begin{notation}
    To define the map $\rm{scr}$, given its connection to projective geometry, it will often be convenient to conflate a vector appearing in a generator of $\St_2$ or $\St_2^2$ with the line it generates. Given $v_1,v_2,v_3,v_4$ nonzero vectors in $F^2$ and taking $\ell_i = \text{span}\{v_i\},$ we will therefore write $[\ell_1,\ell_2]\tensor[\ell_3,\ell_4]$ for the class $[v_1,v_2]\tensor[v_3,v_4]$. By column scaling, there is no ambiguity in this notation.
\end{notation} By \cite[Theorem 23]{CRR}, the group $\St_n\otimes\St_n$ is generated by \emph{Coxeter pairs} $[\ell_1,\compactldots, \ell_n]\otimes [m_1,\compactldots, m_n]$ where $\ell_1,\ldots,\ell_n,m_1,\ldots,m_n \in \bb{P}^{n-1}$ is a collection of lines satisfying (a) $m_1=\ell_1$ and (b) for all $1<i\leq n,$ $m_i\in \ol{\ell_{i-1}\ell_i}$ (where $\ol{\ell_{i-1}\ell_i}$ is the line in $\bb{P}^{n-1}$ given by span of $\ell_{i-1}$ and $\ell_i$). We call such an element \emph{generic} if (b') for all $1<i\leq n,$ $m_i\in \ol{\ell_{i-1}\ell_i} \setminus \{\ell_{i-1},\ell_i\}$, and consider the set
\[X_n \coloneq \{(v_1,\dots, v_n, w_1,\dots, w_n)\in (\bbP^{n-1})^{2n}\mid [v_1,\compactldots, v_n]\otimes [w_1,\compactldots, w_n] \text{ is generic}\}.\]

\begin{lemma}
    The set $X_n/\SL_n$ of orbits of generic elements is in bijection with $F^\times/(F^\times)^n$.
\end{lemma}

\begin{proof}
    As the group $\SL_n$ acts transitively on the first $n$ terms, $\SL_n$ carries any element of $X_n$ to a tuple of the form 
    \[\scr{J}(x_1,\dots, x_{n-1}) \coloneq (e_1,\dots, e_n, e_1, e_2+x_1e_1 , \dots, e_n+x_{n-1}e_{n-1})\]
    where each $x_i\in F^{\times}$.  Since the stabilizer of $(e_1,\dots, e_n)$ is the group of diagonal matrices in $\SL_n$, two such tuples represent the same orbit if and only if there is a diagonal matrix $\rm{diag}(\alpha_1,\dots,\alpha_n)$ with $\alpha_1 \cdots \alpha_n = 1$ that carries one to the other. Since $\rm{diag}(\alpha_1,\dots,\alpha_n) \cdot \scr{J}(x_1,\dots, x_{n-1}) = \scr{J}(\frac{\alpha_1}{\alpha_2} x_1 ,\dots, \frac{\alpha_{n-1}}{\alpha_n} x_{n-1})$, this means that $\scr{J}(x_1,\dots, x_{n-1})$ and $\scr{J}(y_1,\dots, y_{n-1})$ represent the same orbit if and only if there are units $\alpha_1,\dots, \alpha_n\in F^\times$ solving the system of $n$ equations
    \[\frac{\alpha_1}{\alpha_2} = \frac{y_1}{x_1}, \quad \ldots \quad , \frac{\alpha_{n-1}}{\alpha_n} = \frac{y_{n-1}}{x_{n-1}}, \quad \text{and} \quad \alpha_1\cdots\alpha_n =1.\]
    It is elementary to check that the latter holds if and only if an $n$th root
    \begin{equation}
        \label{eq:nthroot}
        \alpha_n = \sqrt[n]{\frac{x_1 x_2^2 \dots x_{n-1}^{n-1}}{y_1y_2^2\dots y_{n-1}^{n-1}}} \text{ exists in } F.
    \end{equation}
    It follows that there is a well-defined map $F^\times/(F^\times)^n \to X_n/\SL_n$ sending $[\beta]$ to the orbit of $\scr{J}(\beta, 1,\dots, 1)$, since for all $\lambda^n \in (F^\times)^n$ an $n$th root of $\beta / (\beta \cdot \lambda^n)$ is given by $\lambda^{-1}$. Its inverse is the map that sends the orbit represented by $\scr{J}(x_1,\dots, x_{n-1})$ to $[x_1 x_2^2 \dots x_{n-1}^{n-1}]$, which is well-defined because $(y_1y_2^2\dots y_{n-1}^{n-1}) \cdot \alpha_n^n = x_1 x_2^2 \dots x_{n-1}^{n-1}$ whenever \eqref{eq:nthroot} holds.
\end{proof}

We denote the quotient map sending a generic element to its orbit by 
\[q \colon X_n \longrightarrow F^\times/(F^\times)^n.\]
This is a generalization of the reduced cross-ratio: for $n=2$, one observes that $q(\ell_1,\ell_2,\ell_1,\ell_3)$ yields the reduced cross-ratio $o(\ell_1,\ell_2,\ell_3)$ of \cite[p.~10]{Barge}. Since the set $X_n$ parametrises a generating set for $\scr{A}_n$ by \cref{theorem: An generated by 1s}, we would like to use $q$ to construct a map on $\scr{A}_n$. To do so, we must understand the relations between generic elements. 

We will next make use of a known resolution of the Steinberg module, due to Lee and Szczarba. Fixing a field $F$ and an integer $n$, we let $K$ be the semi-simplicial set whose set of $k$-simplices is the set of tuples $(v_1, \dots, v_{k+1})$ of nonzero vectors $v_i \in F^n$. Let $L$ be the semi-simplicial subset consisting of those simplices which do \emph{not} span $F^n$. We note that the associated chain complexes satisfy $C_k(K) = C_k(L)$ for $k\leq n-2$. The following chain complex is then a $\bbZ[\GL_n]$-free resolution of $\St_n$ by \cite[Theorem 3.1]{LeeSzczarba}:

\begin{definition}
    \label{def:sharbly}
     The \emph{Sharbly resolution} is given by $P_\bullet\to \St_n$ with $P_i \coloneq C_{i+n-1}(K,L)$ as above, and where the map $\epsilon \colon P_0 =  C_{n-1}(K,L)\to \St_n(F)$ is given by sending a basis $(v_1,\dots, v_n)$ of $F^n$ to the generator $[v_1,\compactldots, v_n]$ of $\St_n(F)$ in \cref{def:steinberg}.
\end{definition}

If $n=2$, whenever three lines $\ell_1,\ell_2,\ell_3$ are pairwise disjoint we must have that $\ell_3\in \ol{\ell_1\ell_2} \setminus \{\ell_1,\ell_2\}$. This fact allows us to determine a spanning set for the relations between generic elements of a quotient of $\St_2\otimes\St_2$. 
We introduce the notation
\[\ol{\St^2_2} \coloneq \dfrac{\St^2_2}{\text{span}\left\{[\ell_1,\ell_2]\otimes[\ell_1,\ell_2]\mid \ell_1,\ell_2\in \bbP^1\right\}}.\]

\begin{remark}In fact, one can think of this as the quotient of $\St^2_2$ by ``decomposables'', with respect to the product structure on double Steinberg modules from \cite[Section 3.3]{CRR} (distinct from the one considered above) and from that perspective the first statement in the next lemma is \cite[Proposition 27]{CRR}.\end{remark}

For nonzero $v_1,v_2,v_3,v_4\in F^4$, define an element of $P_0\otimes P_0$ by:
    \[r(v_1,v_2,v_3,v_4) \coloneq (v_1,v_2)\otimes(v_1,v_3) - (v_1,v_2)\otimes (v_1,v_4) + (v_1,v_3)\otimes(v_1,v_4)-(v_2,v_3)\otimes (v_2,v_4).\]

\begin{lemma}\label{lem:dim2Rels}
    $\ol{\St^2_2}$ is generated by generic elements, with relations generated by the classes $\epsilon(r(v_1,v_2,v_3,v_4))$.
\end{lemma}

\begin{proof}
    First, we verify that each $r(v_1,v_2,v_3,v_4)$ maps to zero in $\smash{\ol{\St^2_2}}$. To this end, we work with the following model of the double Steinberg module: $\St^2_2 = \ker(\partial_1) \subseteq C_1(T(F^2)\ast T(F^2))$ where $T(F^2)$ is the Tits building (i.e.\ the set of lines in $F^2$). Denoting the line in $F^2$ spanned by $v_i$ with $\ell_i$ and writing $e_{\ell_i,\ell_j}$ for the $1$-cell in $T(F^2) \ast T(F^2)$ corresponding to $\ell_i \ast \ell_j$, we have:
    \begin{align*}
		\epsilon(r(v_1,v_2,v_3,v_4)) &= [\ell_1,\ell_2]\otimes [\ell_1,\ell_3] - [\ell_1,\ell_2]\otimes[\ell_1,\ell_4] + [\ell_1,\ell_3]\otimes[\ell_1,\ell_4] -[\ell_2,\ell_3]\otimes[\ell_2,\ell_4]\\
		&= (e_{\ell_1,\ell_1} - e_{\ell_1,\ell_3}-e_{\ell_2,\ell_1} + e_{\ell_2,\ell_3}) - (e_{\ell_1,\ell_1} - e_{\ell_1,\ell_4}-e_{\ell_2,\ell_1} + e_{\ell_2,\ell_4})\\ &\qquad + ( e_{\ell_1,\ell_1} - e_{\ell_1,\ell_4}-e_{\ell_3,\ell_1} + e_{\ell_3,\ell_4}) -(e_{\ell_2,\ell_2} - e_{\ell_2,\ell_4}-e_{\ell_3,\ell_2} + e_{\ell_3,\ell_4})\\
		&= e_{\ell_1,\ell_1} - e_{\ell_1,\ell_3}-e_{\ell_2,\ell_1} + e_{\ell_2,\ell_3} - e_{\ell_1,\ell_1} + e_{\ell_1,\ell_4}+e_{\ell_2,\ell_1} - e_{\ell_2,\ell_4}\\ &\qquad +  e_{\ell_1,\ell_1} - e_{\ell_1,\ell_4}-e_{\ell_3,\ell_1} + e_{\ell_3,\ell_4} -e_{\ell_2,\ell_2} + e_{\ell_2,\ell_4}+e_{\ell_3,\ell_2} - e_{\ell_3,\ell_4}\\
		&=  - e_{\ell_1,\ell_3} + e_{\ell_2,\ell_3}  +  e_{\ell_1,\ell_1} -e_{\ell_3,\ell_1}  -e_{\ell_2,\ell_2} +e_{\ell_3,\ell_2} \\
		&= [\ell_1,\ell_3]\otimes[\ell_1,\ell_3] - [\ell_2,\ell_3]\otimes[\ell_2,\ell_3].
	\end{align*}
    Therefore $\epsilon(r(v_1,v_2,v_3,v_4))$ is a relation for $\ol{\St^2_2}$. We now demonstrate that these classes generate all relations between generic elements.

    Let $F_\bullet\twoheadrightarrow \St^2_2$ be the resolution obtained by tensoring the Sharbly resolution with itself. Let $G\subseteq F_0 = P_0\otimes P_0$ be the submodule spanned by classes of the form $(v_1,v_2)\otimes (v_1, v_3)$ -- note that we may have $v_3=\lambda v_2$, but that $\{v_1,v_2\}$ and $\{v_1,v_3\}$ each must be a basis of $F^2$. Let us write \[R \coloneq d(F_1)\oplus \text{span}\left\{(v_1,v_2)\otimes (v_1,\lambda v_2) \mid v_1,v_2\in F^2; \;\lambda\in F^\times\right\}.\] The map $\epsilon|_G \colon G \to \St^2_2$ is a surjection, since its image contains the spanning set given by the Solomon--Tits theorem. Therefore the map
    \[\epsilon' \colon G\longrightarrow \ol{\St^2_2}\] 
    is surjective as well and the relations of $\ol{\St^2_2}$ are given by $\ker(\epsilon'\mid_G) = R\cap G$. To calculate $R \cap G$, we consider the map 
    \begin{align*}\sigma \colon F_0 &\longrightarrow G \\
    (v_1,v_2)\otimes (v_3,v_4) &\longmapsto (v_1,v_2)\otimes (v_1,v_4) - (v_1,v_2)\otimes (v_1,v_3).\end{align*}
    It is straightforward to verify this is well-defined, $\SL_2$-equivariant, and satisfies $\sigma|_G = \id_G$. Furthermore, it holds that \[(\id-\sigma)((v_1,v_2)\otimes (v_3,v_4)) = d((v_1,v_2)\otimes (v_1,v_3,v_4)),\] which implies that $(\id-\sigma)(x)\in R.$ Therefore, if $x\in R$, we have $\sigma(x)\in R\cap G$ and hence $\sigma$ splits the inclusion of $R\cap G$ into $R$. A generating set of $R \cap G$ is thus:
    \begin{align*}
        \sigma((v_1,v_2)\otimes ( v_1,\lambda v_2))&= (v_1,v_2)\otimes (v_1, \lambda v_2)\\
        \sigma(d((v_1,v_2,v_3)\otimes (v_4,v_5))) &=  \sigma( (v_1,v_2)\otimes(v_4,v_5)-(v_1,v_3)\otimes (v_4,v_5)+(v_2,v_3)\otimes (v_4,v_5))\\
        &= (v_1,v_2)\otimes(v_1,v_5)-(v_1,v_2)\otimes(v_1,v_4))-(v_1,v_3)\otimes(v_1,v_5) \\ & \qquad + (v_1,v_3)\otimes(v_1,v_4)+(v_2,v_3)\otimes(v_2,v_5) - (v_2,v_3)\otimes(v_2,v_4)\\
        &= r(v_1,v_2,v_3,v_4) - r(v_1,v_2,v_3,v_5).\\
        \sigma(d((v_1,v_2)\otimes(v_3,v_4,v_5))) &= \sigma((v_1,v_2)\otimes(v_4,v_5) - (v_1,v_2)\otimes(v_3,v_5) + (v_1,v_2)\otimes(v_3,v_4))\\
        &= (v_1,v_2)\otimes(v_1,v_5) -(v_1,v_2)\otimes(v_1,v_4) - (v_1,v_2)\otimes(v_1,v_5)\\ & \qquad + (v_1,v_2)\otimes(v_1,v_3) + (v_1,v_2)\otimes(v_1,v_4) - (v_1,v_2)\otimes(v_1,v_3)\\
        &= 0.
    \end{align*}
    Since $r(v_1,v_2,v_3,v_4)$ is a relation, $R\cap G$ is generated by $r(v_1,v_2,v_3,v_4)$ along with $(v_1,v_2)\otimes(v_1,\lambda v_2)$. However, these last relations are redundant:
    \begin{align*}
    \epsilon(r(v_1,v_2,\lambda v_2,v_4)) &= \epsilon((v_1,v_2)\otimes(v_1,\lambda v_2) - (v_1,v_2)\otimes(v_1,v_4)\\ &\qquad + (v_1,\lambda v_2)\otimes(v_1,v_4) -(v_2,\lambda v_2)\otimes(v_2,v_4))\\
    &= [\ell_1,\ell_2]\otimes[\ell_1,\ell_2] \\ &= \epsilon((v_1,v_2)\otimes(v_1,\lambda v_2)).
    \end{align*}
\end{proof}

\begin{lemma}There is a well-defined map
\begin{align*} \rm{scr} \colon \scr{A}_2 &\longrightarrow \rm{W}(F) \\
[\ell_1,\ell_2] \otimes [\ell_1,\ell_3] &\longmapsto \begin{cases} \langle q(\ell_1,\ell_2,\ell_1,\ell_3) \rangle & \text{if $[\ell_1,\ell_2] \otimes [\ell_1,\ell_3]$ is generic,} \\
0 & \text{else.}\end{cases}\end{align*}
\end{lemma}

\begin{proof} In \cite[Proposition 1.1]{Barge}, Barge proved that a map
\begin{align*} \varphi \colon (\bb{P}^1)^3 &\longrightarrow \rm{W}(F) \\
(\ell_1,\ell_2,\ell_3) &\longmapsto \begin{cases} o(\ell_1,\ell_2,\ell_3) & \text{if the three lines are distinct,} \\
0 & \text{else,}\end{cases}
\end{align*}
is a 2-cocycle, i.e.~$\varphi(\ell_2,\ell_3,\ell_4)-\varphi(\ell_1,\ell_3,\ell_4)+\varphi(\ell_1,\ell_2,\ell_4)-\varphi(\ell_1,\ell_2,\ell_3) = 0$. Recall that $q(\ell_1,\ell_2,\ell_1,\ell_3) = o(\ell_1,\ell_2,\ell_3)$ for a generic $(\ell_1,\ell_2,\ell_1,\ell_3)$. We now observe that the formula in the statement factors over the $\SL_2$-coinvariants of $\smash{\ol{\St^2_2}}$. Then \cref{lem:dim2Rels} says that this is spanned by generic elements satisfying only the relations $r(v_1,v_2,v_3,v_3)$ and by the 2-cocycle property $\varphi$ sends this to zero.
\end{proof}

Having constructed both maps, we now prove of the main result of this section:

\begin{proof}[Proof of Theorem \ref{mainthm:dimension 2}] The maps described above fit into a commutative diagram
\[\begin{tikzcd}
	\scr{A}_2 = (\St_2 \otimes \St_2)_{\rm{SL}_2} \arrow[dashed]{rd}[description]{\Phi} \arrow[bend left=15]{rrd}{\epsilon} \arrow[bend right=15]{rdd}[swap]{\rm{scr}} & &[15pt] \\[-7pt]
	& \rm{GW}(F) \rar[swap]{\rm{rk}} \dar{\rm{pr}} & \bb{Z} \dar{\rm{pr}} \\[-5pt]
	& \rm{W}(F) \rar[swap]{\rm{rk}\,(\rm{mod}\,2)} & \bb{Z}/2
\end{tikzcd}\]
because $\epsilon(\{a\}_2) = 1$ and $\rm{rk}(\rm{scr}(\{a\}_2)) \equiv 1$ since the form $\langle a \rangle$ has 1-dimensional underlying vector space. Since the commutative square in this diagram is a pullback, $\epsilon$ and $\rm{scr}$ induce a unique dashed map $\Phi$. We prove it is an isomorphism by constructing an inverse $\Psi$, by declaring it satisfies
	\begin{align*} \Psi \colon \rm{GW}(F) &\longrightarrow \scr{A}_2 \\
    \langle x \rangle &\longmapsto \{x\}_2.\end{align*} 
Once we establish this is well-defined, it is an isomorphism because it is surjective by \cref{theorem: An generated by 1s} and satisfies $\Phi \Psi = \rm{id}_{\rm{GW}(F)}$ so is also injective. 

To see $\Psi$ is well-defined, we must show that it takes the relations of $\rm{GW}(F)$ to relations in $\scr{A}_2$. The former are given by (i) $\langle\lambda^2 a \rangle = \langle a \rangle$ and (ii) $\langle a \rangle + \langle b \rangle = \langle a+b\rangle + \langle ab(a+b) \rangle$. For relation (i), this is simply \cref{lem:f-times-action} \eqref{enum:f-times-action-ii}.  For relation (ii), we first observe that column addition yields
\begin{align*}\langle \begin{bsmallmatrix} 1 & 0 \\ 0 & 1  \end{bsmallmatrix} \rangle &= \langle \begin{bsmallmatrix} 1 & a \\ 0 & 1  \end{bsmallmatrix} \rangle + \langle \begin{bsmallmatrix} 1 & 0 \\ 1/a & 1  \end{bsmallmatrix} = \langle \begin{bsmallmatrix} 1 & a \\ 0 & 1  \end{bsmallmatrix} \rangle + \langle \begin{bsmallmatrix} 1 & -1/a \\ 0 & 1  \end{bsmallmatrix} \rangle \\
&= \{a\}_2+\{-1/a\}_2 = \{a\}_2+\{-a\}_2\end{align*}
where \cref{lem:f-times-action} \eqref{enum:f-times-action-ii} is used in the last step.
Next consider the equation obtained by row addition
\begin{align*}\langle \begin{bsmallmatrix} 1 & 1 \\	a & a+b \end{bsmallmatrix} \rangle &= \langle \begin{bsmallmatrix} 1 & 1 \\ 0 & b \end{bsmallmatrix} \rangle + \langle \begin{bsmallmatrix} 0 & -b/a \\ a & a+b \end{bsmallmatrix} \rangle = \langle \begin{bsmallmatrix} 1 & 1 \\ 0 & b \end{bsmallmatrix} \rangle -\langle \begin{bsmallmatrix} 1 & a+b \\0 & b/a \end{bsmallmatrix} \rangle  \\
&= \{b\}_2-\{a/b(a+b)\}_2 = \{b\}_2-\{ab(a+b)\}_2\end{align*} and the equation obtained by column addition
\begin{align*}\langle \begin{bsmallmatrix} 1 & 1 \\	a & a+b \end{bsmallmatrix} \rangle &= \langle \begin{bsmallmatrix} 1 & 0 \\ a & b \end{bsmallmatrix} \rangle + \langle  \begin{bsmallmatrix} 0 & 1 \\	-b & a+b \end{bsmallmatrix} \rangle = \langle \begin{bsmallmatrix} 1 & 0 \\ a & b \end{bsmallmatrix} \rangle - \langle  \begin{bsmallmatrix} 1 & -(a+b) \\	0 & 1 \end{bsmallmatrix} \rangle \\
&= \{-a\}_2-\{-(a+b)\}_2 = -\{a\}_2+\{a+b\}_2\end{align*}
where in the last step we used $\{-a\}_2= \langle \begin{bsmallmatrix} 1 & 0 \\ 0 & 1  \end{bsmallmatrix} \rangle-\{a\}_2$ and similarly for $\{-(a+b)\}_2$, both copies of $\langle \begin{bsmallmatrix} 1 & 0 \\ 0 & 1  \end{bsmallmatrix} \rangle$ canceling. Since both expressions are equal, we obtain the result.
\end{proof}

\section{Theorems B and C: higher rank} We now prove \cref{mainthm:higherdim} and \cref{mainthm:rational}, starting with the former.

\subsection{Proof of \cref{mainthm:higherdim}} Part \eqref{enum:higherdim-i} of \cref{mainthm:higherdim} says that the $\bb{Z}[F^{\times}]$-module $\scr{A}_n$ is a quotient of $\bb{Z}[F^\times/(F^\times)^n]$. To prove this, one combines \cref{theorem: An generated by 1s}, which says that $\scr{A}_n$ is generated by $\{1\}_n$ as a $\bbZ[F^\times]$-module, and \cref{lem:f-times-action} \eqref{enum:f-times-action-ii}, which says that $\{\lambda^n 1\}_n = \{1\}_n$.

For part \eqref{enum:higherdim-ii} of \cref{mainthm:higherdim} we recall there is a surjective map
\[\scr{A}_n = \St^2_n(F)_{\SL_n(F)} \longrightarrow \St^2_n(F)_{\GL_n(F)} \cong \bb{Z},\]
split by the map sending $1 \in \bb{Z}$ to $\{1\}_n$. To prove that the complementary summand is $n$-torsion when $F$ contains an $n$th root of $-1$, we use that $\{1\}_n = \{-1\}_n$ and so by \cref{exam:products-odd} we have $\{1\}_{n-1}\{1\}_1 = n \{1\}_n$ in this case. Combined with the equality $\{a\}_1 = \{1\}_1$ from \cref{lem:f-times-action} \eqref{enum:f-times-action-ii}, we find that for $a \in F^\times$ it holds that
\[n \{a\}_n = a \cdot n\{1\}_n = a \cdot \{1\}_{n-1} \{1\}_1 = \{1\}_{n-1} \{a\}_1 = \{1\}_{n-1}\{1\}_1 = n \{1\}_n.\]
It follows that $n(\{a\}_n-\{1\}_n) = 0$, which completes the proof of \cref{mainthm:higherdim}.

\subsection{Proof of \cref{mainthm:rational}} In this section, we study the rationalization $\scr{A}^\bb{Q} \coloneq \scr{A}\otimes_\bbZ\bb{Q}$ of the associative nonunital $\bb{Z}[F^\times]$-algebra $\scr{A}$. We first prove that as a graded $\bb{Q}[F^\times]$-algebra, $\scr{A}^\bb{Q}$ is generated by $\scr{A}^\bb{Q}_1$ and $\scr{A}^\bb{Q}_2$.

By \cref{theorem: An generated by 1s}, $\scr{A}^\bb{Q}_n$ is generated by the elements $\{1\}_n$ as a $\bb{Q}[F^\times]$-module. Thus it suffices to show that each $\{1\}_n$ can be expressed as a linear combination of products of elements in degrees $1$ and $2$. Recall from \cref{thm: products} that \[
        \{1\}_a\cdot \{1\}_b = P(a+b,a)\{1\}_{a+b}+Q(a+b,a)\{-1\}_{a+b}.
    \] 
    Let $n\geq 3$. Taking $a=2$ and $b=n-2$ gives \begin{equation}\label{eq:1}\{1\}_{2}\{1\}_{n-2}=P(n,2)\{1\}_n+Q(n,2)\{-1\}_n.\end{equation}
    There are two cases:
    \begin{enumerate}[{Case} 1.]
        \item The first case is that $n=2m+1$ is odd and we have $\{-1\}_n=\{1\}_n$ by \cref{exam:products-odd}. Then \eqref{eq:1} becomes \[\{1\}_{2m-1}\{1\}_2=\binom{2m+1}{2}\{1\}_{2m+1}=m(2m+1)\{1\}_{2m+1}.\] Thus we have $\{1\}_{2m+1}=\frac{1}{m(2m+1)}\{1\}_{2}\{1\}_{2m-1}$ (using that we rationalized). By induction, this gives the answer for all odd $n$, using that $\{1\}_3=\frac{1}{3}\{1\}_1\{1\}_2$.
        \item The second case is that $n=2m$ is even, where we cannot assume $\{-1\}_n = \{1\}_n$. We consider the product $\{-1\}_2\{1\}_{n-2}$ and calculate it using \eqref{eq:1} as follows
    \[
        \{-1\}_2\{1\}_{n-2}= (-1)^{-1} \cdot (\{1\}_2\{1\}_{n-2}) = P(n,2)\{-1\}_n+Q(n,2)\{1\}_n.
    \]
    Evaluating $P(n,2)(\{1\}_2\{1\}_{n-2}) - Q(n,2)(\{-1\}_2\{1\}_{n-2})$ using this and the formula for the product $\{1\}_{2}\{1\}_{n-2}$ in \eqref{eq:1} leads to the equality (writing $n = 2m$)
    \[\{1\}_{2m}=\frac{P(2m,2)\{1\}_2-Q(2m,2)\{-1\}_2}{P(2m,2)^2-Q(2m,2)^2}\{1\}_{2m-2}.\]
    We substitute into this the expressions from \cref{lem:expressions-of-p-and-q}, given by \[\qquad \qquad P(2m,2)=\frac{1}{2}\left(\binom{2m}{2}+\binom{m}{1}\right)=m^2, \quad \text{and} \quad Q(2m,2)=\binom{2m}{2}-m^2=m(m-1),\]
    to obtain \[\{1\}_{2m}=\frac{m\{1\}_2-(m-1)\{-1\}_2}{m(2m-1)}\{1\}_{2m-2}.\]
    Finally, \cref{thm: products} implies $\{1\}_1^2=\{1\}_2+\{-1\}_2$. Substituting this into the previous equation yields \[\{1\}_{2m}=\frac{(2m-1)\{1\}_2-(m-1)\{1\}_1^2}{m(2m-1)}\{1\}_{2m-2}.\]
    This gives a recursion for all even-degree elements, concluding that all even-degree elements are generated by $\{1\}_1$ and $\{1\}_2$.
    \end{enumerate}

\medskip

\noindent This tells us that the (underived) indecomposables are nonzero only in ranks $n=1,2$. For degree reasons, those in rank 1 are given by $\smash{\scr{A}_1^\bb{Q}} \cong \bb{Q}$. In rank 2, we compute that the indecomposables are given by the quotient of $\smash{\scr{A}_2^\bb{Q}}$ by span of $\{1\}_1 \cdot \{1\}_1 = 2\{1\}_2$ and this agrees under the isomorphism $\smash{\scr{A}_2^\bb{Q}} \cong \rm{GW}(F) \otimes \bb{Q}$ with the quotient of $\rm{GW}(F) \otimes \bb{Q}$ by the span of the 1-dimensional form $\langle 1 \rangle$. Since the augmentation $\rm{GW}(F) \otimes \bb{Q} \to \bb{Q}$ is split by sending $1$ to $\langle 1 \rangle$, it follows that $\smash{\scr{A}_2^\bb{Q}}$ is isomorphic to $I(F) \otimes \bb{Q}$. 

\begin{remark}Since rationalizing preserves pullbacks and $\bb{Z}/2 \otimes \bb{Q} = 0$, it follows that there is an isomorphism $I(F) \otimes \bb{Q} \cong \rm{W}(F) \otimes \bb{Q}$.\end{remark}

\section{Theorem D: $E_k$-cells} \label{sec:ek-cells}
Recent work of Galatius--Kupers--Randal-Williams considers an $E_\infty$-algebra $\bf{BGL}(F)_\bb{Z}$ constructed out of the general linear groups $\GL_n(F)$ and proves that for any infinite field $F$ there is an isomorphism
\[H^{E_2}_{n,d}(\bf{BGL}(F)_\bb{Z}) \cong \widetilde{H}_{d-(2n-2)}(\GL_n(F);\St^2_n(F)),\]
and so the $\GL_n(F)$-coinvariants of $\St^2_n(F)$ give the first nonzero $E_2$-homology groups of $\bf{BGL}(F)_\bb{Z}$ \cite[Theorem 6.5]{GKRW25}. The goal of this section is to adapt their argument to the case of special linear groups and prove \cref{mainthm:ekcells}. See \cite{GKRW23,GKRW25,KRS1} for more details.

\medskip

Our starting point is the 1-groupoid $\rm{Vect}_F$ whose objects are finite-dimensional vector spaces over $F$ and whose morphisms are the linear isomorphisms. Direct sum makes this a symmetric monoidal category. To equip the collection of special linear groups with an $E_\infty$-algebra structure, we consider a functor
\[\rm{grdet} \colon \rm{Vect}_F \longrightarrow \rm{grDet}_F.\]
Here, $\rm{grDet}_F$ is the groupoid with objects given by nonnegative integers and automorphisms of $n$ are given by $F^\times$, and the functor $\rm{grdet}$ takes a vector space $V$ to $\dim(V)$ and a linear automorphism $A$ to $\det(A)$. We encourage the reader to think of the object $n$ of $\rm{grDet}_F$ as corresponding to $\Lambda^n F^n$, which for example, elucidates the following. There is a symmetric monoidal structure on $\rm{grDet}_F$ that makes $\rm{grdet}$ strong symmetric monoidal: its tensor product is on objects given by addition, on morphisms by multiplication, and its symmetry isomorphism $n+m \simeq m+n$ is $(-1)^{mn}$. 

Let $\rm{Spc}$ denote the $\infty$-category of spaces and $\scr{D}_\bb{Z}$ the derived $\infty$-category of $\bb{Z}$-modules. For any presentable symmetric monoidal category $\scr{C}$, such as $\rm{Spc}$ or $\scr{D}_\bb{Z}$, Day convolution then endows the domain and target of the left Kan extension
\[\rm{grdet}_! \colon \rm{Fun}(\rm{Vect}_F,\scr{C}) \longrightarrow \rm{Fun}(\rm{grDet}_F,\scr{C})\]
with symmetric monoidal structures so that $\rm{grdet}_!$ is strong symmetric monoidal. If we take the terminal unital $E_\infty$-algebra given by the constant functor $\ul{\ast} \in \rm{Fun}(\rm{Vect}_F,\rm{Spc})$, apply the symmetric monoidal functor $C_*(-;\bb{Z})$ of integral chains, and then pass to the augmentation ideal to obtain a nonunital $E_{\infty}$-algebra $\ul{\bb{Z}}_{>0} \in \rm{Fun}(\rm{Vect}_F,\scr{D}_\bb{Z})$, then 
\[\bf{BSL}(F)_\bb{Z} \coloneq \rm{grdet}_!(\ul{\bb{Z}}_{>0})\]
is a nonunital $E_\infty$-algebra with additional $\bb{N}$-grading satisfying $\bf{BSL}(F)_\bb{Z}(n) \simeq C_*(\rm{BSL}_n(F);\bb{Z})$ for $n > 0$ (it vanishes for $n=0$). We note that the action of the automorphisms in $\rm{grDet}_F$ endows its homology groups with the usual $\bb{Z}[F^\times]$-action.

Now fix $1 \leq k \leq \infty$. The \emph{trivial algebra} functor $\rm{triv}_{E^\rm{nu}_k}$ which assigns to an object $X \in \rm{Fun}(\rm{grDet}_F,\scr{C})$ the trivial nonunital $E_k$-algebra with underlying object $X$, admits a left adjoint $\rm{cot}_{E^\rm{nu}_k}$ of \emph{$E_k$-indecomposables} (we follow the more modern notation of \cite{KRS1}, noting that $\smash{\rm{cot}_{E^\rm{nu}_k}}$ is denoted $\smash{Q^{E_k}_\bb{L}}$ in \cite{GKRW23,GKRW25}). For a nonunital $E_k$-algebra $\bf{A} \in \rm{Alg}_{E_k^\rm{nu}}(\rm{Fun}(\rm{grDet}_F,\scr{D}_\bb{Z}))$, we define its $E_k$-homology groups as
\[H^{E_k}_{n,d}(\bf{A}) \coloneq \pi_d\big(\rm{cot}_{E_k^\rm{nu}}(\bf{A})(n)\big).\]
To compute these, we use that one can interchange the constructions $\rm{grdet}_!$ and $C_*(-;\bb{Z})$ with $E_k$-indecomposables $\rm{cot}_{E^\rm{nu}_k}$, as they are symmetric monoidal left adjoints: We let $\underline{\ast}_{> 0} \in \rm{Alg}_{E_k^\rm{nu}}(\rm{Fun}(\rm{Vect}_F,\rm{Spc}))$ be the trivial non-unital $E_k$-algebra and combine \cite[Proposition 17.14]{GKRW23} with the proof of \cite[Theorem 5.20]{GKRW25} to see that there are equivariant homotopy equivalences
\[
    S^k \wedge \rm{cot}_{E^{\rm{nu}}_k}(\underline{\ast}_{> 0})(n) \simeq \widetilde{D}^k(n)
\]
where $\widetilde{D}^k(n)$ is the pointed space given by the $k$-fold split building of \cite[Definition 5.9]{GKRW25}. Then, we use the symmetric monoidal left adjointness to see that
\[
    H^{E_k}_{n,d}(\bf{BSL}(F)_\bb{Z}) \cong \widetilde{H}_{d-k}(\widetilde{D}^k(n) {\sslash} \rm{SL}_n(F);\bb{Z}),
\]
where on the right we take pointed orbits.

We would now like to finish our argument as \cite{GKRW25}: There is a $\rm{GL}_n(F)$-equivariant map $\smash{\widetilde{D}^k(n)} \to D^k(n)$ to a non-split building \cite[Definition 5.4]{GKRW25}. Since $F$ is infinite, by a result of Nesterenko--Suslin the induced map on $\GL_n(F)$-orbits is an equivalence \cite[Theorem 5.18]{GKRW25}. For $k=2$ we have that $ D^2(n) \simeq \Sigma^2 (T(F^n) \wedge \Sigma^2 T(F^n))$ \cite[Lemma 6.1, Proposition 6.3]{GKRW25}. Here $T(F^n)$ is the \emph{Tits building}, which is $(n-2)$-spherical and whose $(n-2)$nd reduced homology is the definition of $\rm{St}_n(F)$. Thus $D^2(n)$ is $2n$-spherical and its reduced $2n$th homology is given by $\St^2_n(F) = \St_n(F) \otimes \St_n(F)$, which implies the claim.

The issue is that the statement of Nesterenko--Suslin does not apply to special linear groups. However, our next lemma shows that its $\SL_n(F)$ analogue is nonethelesss true by a localization trick due to Schlichting \cite[Section 2]{Schlichting}, using that $F$ is infinite.

\begin{lemma}\label{lem:ns-replacement} Let $p,q \geq 1$ such that $p+q = n$ and consider the subgroup $\rm{SAff}_{p,q}(F) \subset \rm{SL}_n(F)$ of matrices of the form $\begin{bsmallmatrix} A & B \\ 0 & \rm{id}_p \end{bsmallmatrix}$. Then the following inclusion induces an isomorphism on homology
\[\SL_q(F) \longrightarrow \rm{SAff}_{p,q}(F).\]
\end{lemma}

\begin{proof}The infinite field $F$ is a ring with many units and thus is an $S(m)$-algebra for every $m \geq 1$ (see \cite[Definition 2.1 et seq.]{Schlichting}). \cite[Theorem 2.5]{Schlichting} says taking $m = l \cdot n$ and $t = n \cdot q$ for $l \geq 1$, there exists an element $s_{m,-t} \in \bb{Z}[F^\times]$ so that the inclusion induces an isomorphism $H_r(\rm{SL}_q(F)) \cong s^{-1}_{m,-t} H_r(\rm{SAff}_{p,q}(F))$ for $r \leq l$. Looking at construction of $s_{m,-t}$ in the beginning of \cite[Section 2]{Schlichting}, we see it is given by a sum in $\bb{Z}[F^\times]$ of $(-t)$th powers of elements $a_J = \sum_{j \in J} a_j$ that satisfy $a_J \in F^\times$ (because $F$ is an $S(m)$-algebra).  Since we picked $t = n \cdot q$, we may represent the action of the $(-t)$th power $a_J^{-t}$ by conjugation with the diagonal matrix $\rm{diag}(a_J^{-t/n},\compactldots,a_J^{-t/n})$ of $\GL_n(F)$. But this is central, and hence acts as identity on both the subgroups $\SL_q(F)$ and $\rm{SAff}_{p,q}(F)$ of $\GL_n(F)$. As the augmentation $\epsilon \colon \bb{Z}[F^\times] \to \bb{Z}$ that takes each generator $a_J \in F^\times$ to $1$ sends $s_{m,-t}$ to $1$ as well (see \cite[Section 2]{Schlichting}), we conclude that $s_{m,-t}$ acts as the identity on $H_*(-;\bb{Z})$ of both sides. Since $r$ was arbitrary, the result follows.
\end{proof}

\begin{lemma}If $F$ is infinite then the map $\widetilde{D}^k(n){\sslash}\SL_n(F) \to D^k(n){\sslash}\SL_n(F)$ of pointed spaces induces an isomorphism on homology.
\end{lemma}

\begin{proof}
This proof is an adaptation of \cite[Proposition 5.13 and Theorem 5.18]{GKRW25}, and we begin by discussing the setting: We fix a direct sum decomposition $F^n = \bigoplus_{x \in X} M_x$, so that any $F$-linear endomorphism $f : F^n \to F^n$ can be written as a matrix with entries $f_{x,y} \in \operatorname{Hom}_F(M_x,M_y)$. For a partial order $\leq$ on the index set $X$, an endomorphism of $F^n$ is called upper triangular with respect to $\leq$ if $f_{x,y} = 0$ unless $x \leq y$. We write $\GL(F^n|\leq) \subset \GL_n(F)$ for the subgroup of elements that are upper-triangular with respect to $\leq$, and define $\SL(F^n|\leq)$ as the intersection of $\GL(F^n|{\leq})$ with $\SL_n(F)$. Finally, we record that if $\preceq$ and $\leq$ are two partial orders on the index set $X$ such that $x \preceq y$ implies $x \leq y$, then there is an inclusion $\SL(F^n|{\preceq}) \hookrightarrow \SL(F^n|{\leq})$. For the identity relation $\simeq$ on $X$, we have that $\SL(F^n|\simeq)$ is the group of diagonal matrices given by the intersection of $\GL(F^n|\simeq) = \prod_x \GL(M_x)$ and $\SL_n(F)$. Our key claim is that the inclusion
\[
    \SL(F^n|\simeq) \to \SL(F^n|\leq)
\]
induces an isomorphism on homology in each fixed degree $r$.

To prove this, we replace the Nesterenko--Suslin property used in the proof of \cite[Proposition 5.13]{GKRW25} by \cref{lem:ns-replacement}: Defining $\preceq_0$ as the partial order $\leq$ on $X$, we can inductively define a sequence of partial orders $\preceq_i$ by picking a maximal element $z_i \in X \setminus \{z_1, \dots, z_{i-1}\}$ with respect to $\preceq_{i-1}$, declaring $\preceq_i$ to agree with $\preceq_{i-1}$ on $X\setminus\{z_i\}$ and letting $z_i$ be incomparable to any other element except itself. Noting that $\preceq_{|X|}$ is precisely the identity relation $\simeq$, we may verify the statement by induction on $\preceq_i$ with trivial base case $\preceq_{|X|}$.
For the induction step, we thus need to compare partial orders $\preceq$ and $\leq$ which agree on $X_{-z} \coloneq X \backslash \{z\}$ for some $z \in X$, such that for any $x \in X_{-z}$ (a) neither $z \preceq x$ nor $x \preceq z$ holds, and (b) $x \not\geq z$ (although $x \leq z$ is allowed). We observe that (b) implies that any $f \in \SL(F^n|{\leq})$ satisfies $f(M_z) \subseteq M_z$. Writing $F^n = M_{-z} \oplus M_z$ with $M_{-z} = \bigoplus_{x \neq z} M_x$, we obtain an isomorphism $M_{-z} \to F^n/M_z$. This leads to a surjection $\SL(F^n|{\leq}) \to \GL(M_{-z}|{\leq}_{|M_{-z}})$ whose kernel is $\SL(F^n\,\rm{fix}\,M_z)$. Similarly, there is a surjection $\SL(F^n|{\preceq}) \to \GL(M_{-z}|{\leq}_{|M_{-z}})$ whose kernel is $\SL(M_z)$. To complete the proof of the key claim, we apply Schlichting's result in the map of Lyndon--Hochschild--Serre spectral sequences for the map of short exact sequences
\[\begin{tikzcd} \SL(M_z) \rar \dar & \SL(F^n|{\preceq}) \rar \dar & \GL(M_{-z}|{\leq}_{|M_{-z}}) \dar[equal] \\[-5pt]
\SL(F^n\,\rm{fix}\,M_z) \rar & \SL(F^n|{\leq}) \rar & \GL(M_{-z}|{\leq}_{|M_{-z}})\end{tikzcd}\]
using that the left vertical map is isomorphic to one of the form $\SL_q(F) \to \rm{SAff}_{p,q}(F)$ and hence an isomorphism on homology by \cref{lem:ns-replacement}.

With this variant of \cite[Proposition 5.13]{GKRW25} in hand, one may now adapt the proof of \cite[Theorem 5.18]{GKRW25} to get the statement.\end{proof}

Taking $k=2$ and continuing the argument as outlined for general linear groups, we obtain for an infinite field $F$ an isomorphism
\[H^{E_2}_{n,d}(\bf{BSL}(F)_\bb{Z}) \cong \widetilde{H}_{d-(2n-2)}(\SL_n(F);\St^2_n(F)).\]
Thus, the $\SL_n(F)$-coinvariants of $\St^2_n(F)$ give the first nonzero $E_2$-homology groups of $\bf{BSL}(F)_\bb{Z}$. This completes the proof of \cref{mainthm:ekcells}.

\begin{remark}
    The main difficulty with adapting the homological stability results of \cite{GKRW25} to special linear groups, as alluded to in the introduction, is that free $E_k$-algebras in $\rm{Fun}(\rm{grDet}_F,\scr{D}_\bb{Z})$ have significantly larger homology groups than those in $\rm{Fun}(\bb{N},\scr{D}_\bb{Z})$ (which appears in its place in \cite{GKRW25}). For example, we have $H_{n,d}(\rm{free}_{E_k^\rm{u}}(1_! \bb{Z});\bb{Z}) \cong H_d(\rm{SM}_n(F);\bb{Z})$ where $\rm{SM}_n(F)$ is the group of special monomial matrices.
\end{remark}

\begin{remark}
    When $R$ is not a field, one may similarly define $E_\infty$-algebras $\mathbf{BGL}(R)$ and $\mathbf{BSL}(R)$. There are cases, e.g.~connected semilocal rings with infinite residue fields \cite{GKRW25}, where one can define variants of $\St_n(R)$ and $\St^2_m(R)$ that play an analogous roles for $\mathbf{BGL}(R)$ as $E_1$-, or $E_2$-algebra. However, in these cases it is generally not true that $\St^2_n(R)$ is isomorphic to $\St_n(R)\otimes\St_n(R)$. If so, we believe that the $\SL_n(R)$-coinvariants of $\St^2_n(R)$ and $\St_n(R)\otimes\St_n(R)$ are both worth studying.
\end{remark}

\bibliographystyle{amsalpha}
\bibliography{./references}

\end{document}